\documentclass[10pt,reqno]{amsart}
\usepackage{graphicx,epsfig}
\usepackage{epstopdf}
\usepackage{pdfpages}
\usepackage{amssymb}
\usepackage{empheq}
\usepackage{cases}
\usepackage{amsthm,amsmath}
\usepackage{caption,lipsum}
\usepackage{stmaryrd}
\usepackage{tabularx}
\usepackage{color}
\usepackage{empheq,float}
\DeclareGraphicsExtensions{.eps}

\newtheorem{theorem}{Theorem}[section]
\newtheorem{lemma}[theorem]{Lemma}

\theoremstyle{definition}

\newtheorem{example}[theorem]{Example}

\theoremstyle{remark}
\newtheorem{remark}[theorem]{Remark}

\newcommand{\fe}{\mathrm{e}}

\newcommand{\eps}{\varepsilon}
\newcommand{\bR}{{\mathbb R}}
\newcommand{\bB}{{\bf B}}
\newcommand{\bT}{{\mathbb T}}
\newcommand{\bN}{{\mathbb N}}
\newcommand{\bZ}{{\mathbb Z}}

\newcommand{\bu}{{\bf u}}
\newcommand{\bv}{{\bf v}}
\newcommand{\bz}{{\bf z}}
\newcommand{\bw}{{\bf w}}
\newcommand{\bp}{{\bf p}}
\newcommand{\bq}{{\bf q}}

\newcommand{\bx}{\mathbf{x}}
\newcommand{\by}{{\mathbf y}}
\newcommand{\bh}{{\mathbf h}}
\newcommand{\br}{{\mathbf r}}

\newcommand{\bE}{\mathbf{E}}

\numberwithin{equation}{section}

\begin{document}

\title[Vlasov-Poisson equation in 3D with strong magnetic field]{Uniformly accurate methods for three dimensional Vlasov equations under strong magnetic field with varying direction}


\author[Ph. Chartier]{Philippe Chartier}
\address{\hspace*{-12pt}Ph.~Chartier: Univ Rennes, INRIA-MINGuS, CNRS, IRMAR-UMR 6625, F-35000 Rennes, France}
\email{Philippe.Chartier@inria.fr}

\author[N. Crouseilles]{Nicolas Crouseilles}
\address{\hspace*{-12pt}N.~Crouseilles: Univ Rennes, INRIA-MINGuS, CNRS, IRMAR-UMR 6625, F-35000 Rennes, France}
\email{nicolas.crouseilles@inria.fr}

\author[M. Lemou]{Mohammed Lemou}
\address{\hspace*{-12pt}M.~Lemou: Univ Rennes, CNRS,  INRIA-MINGuS, IRMAR-UMR 6625, F-35000 Rennes, France}
\email{mohammed.lemou@univ-rennes1.fr}

\author[F. M\'{e}hats]{Florian M\'{e}hats}
\address{\hspace*{-12pt}F.~M\'{e}hats: Univ Rennes,  INRIA-MINGuS, CNRS, IRMAR-UMR 6625, F-35000 Rennes, France}
\email{florian.mehats@univ-rennes1.fr}

\author[X. Zhao]{Xiaofei Zhao}
\address{\hspace*{-12pt}X.~Zhao: School of Mathematics and Statistics, Wuhan University, 430072 Wuhan, China; Univ Rennes,  INRIA-MINGuS, CNRS, IRMAR-UMR 6625, F-35000 Rennes, France}
\email{zhxfnus@gmail.com}

%

\date{}

\dedicatory{}

\begin{abstract}
In this paper, we consider the three dimensional Vlasov equation with an inhomogeneous, varying direction, strong magnetic field. Whenever the magnetic field has constant intensity, the oscillations generated by the stiff term are periodic. The homogenized model is then derived and several state-of-the-art multiscale methods, in combination with the Particle-In-Cell discretisation, are proposed for solving the Vlasov-Poisson equation. Their accuracy as much as their computational cost remain essentially independent of the strength of the magnetic field. The proposed schemes thus allow large computational steps, while the full gyro-motion can be restored by a linear interpolation in time. In the linear case, extensions are introduced for general magnetic field (varying intensity and direction). Eventually, numerical experiments are exposed to illustrate the efficiency of the methods and some long-term simulations are presented.
 \\ \\
{\bf Keywords:} Vlasov-Poisson equation, Three dimensions, Strong magnetic field, Varying direction, Uniformly accurate method, Particle-In-Cell. \\ \\
{\bf AMS Subject Classification:} 65L05, 65L20, 65L70.
\end{abstract}

\maketitle

\section{Introduction}
Vlasov models have been widely considered for modelling the dynamics of plasmas as encountered in magnetic fusion devices known as a tokamaks, where a strong external magnetic field is applied so as to confine the charged
particles. In this paper, we consider the three dimensional Vlasov-Poisson equation with a strong non-homogeneous magnetic field whose direction may vary  \cite{degond,Golse,SonnendruckerBook}
\begin{subequations}\label{eq:1}
\begin{align}
&\partial_t f^\eps(t,\bx, \bv) +\bv\cdot\nabla_\bx f^\eps(t,\bx, \bv)+\left(\bE(t,\bx)+\frac{1}{\eps}\bv\times \bB(\bx)\right)\cdot\nabla_\bv f^\eps(t,\bx, \bv)=0,\\
&\nabla_\bx\cdot \bE(t,\bx)=\int_{\bR^3}f^\eps(t,\bx,\bv)d\bv-n_i,\\
&f^\eps(0,\bx,\bv)=f_0(\bx,\bv),
\end{align}
\end{subequations}
where, for a given $T>0$,
$$
f^\eps: (t,\bx, \bv) \in [0,T] \times \bR^3 \times \bR^3 \mapsto f^\eps(t,\bx, \bv) \in \bR
$$
is the unknown,
$$f_0: (\bx, \bv) \in \bR^3 \times \bR^3 \mapsto f_0(\bx, \bv) \in \bR
$$
a given initial distribution, where
$$
\bB: \bx \in \bR^3 \mapsto \bB(\bx)\in \bR^3
$$
denotes the external magnetic field,
$$
\bE: (t,\bx) \in \bR^+\times\bR^3 \mapsto \bE(t,\bx) \in \bR^3
$$ the self-consistent electric-field function, $0<\eps\leq1$ a dimensionless parameter inversely proportional to the strength of the magnetic field and  $n_i\geq0$ the ion density of the background.  The system \eqref{eq:1} has a lot of invariants and
we will be interested in particular in the Hamiltonian defined by 
\begin{equation}\label{energy}
{\mathcal H}(t):=\int_{\bR^3}\int_{\bR^3}\frac{1}{2}|\bv|^2f^\eps(t,\bx,\bv)d\bx d\bv+\frac{1}{2}\int_{\bR^3}|\bE(t,\bx)|^2d\bx.
\end{equation}

The above Vlasov-Poisson model (\ref{eq:1}) is derived from the three dimensional Vlasov-Maxwell equations by considering the electrostatic approximation. Unlike some asymptotically reduced models such as the gyrokinetic equations \cite{Michel, straymond} or the drift-kinetic limit equations \cite{bostan, Sonnendrucker1, filbet},  model (\ref{eq:1}) is set at the kinetic scale and is of paramount importance for studying the plasma dynamics in the tokamak device.

In the strong magnetic field limit regime, the charged particles exhibit very fast rotations with cyclotron period proportional to $\eps$,  while remaining confined along the magnetic line. In such a case, the small parameter $0<\eps\ll1$ renders the solution $f^\eps(t,\bx,\bv)$ of (\ref{eq:1}) highly-oscillatory in time. Classical numerical integrators such as splitting  or finite-difference schemes thus require time steps smaller than the cyclotron period in order to accurately capture the dynamics, thus implying severe computational burden. Recent efforts  have aimed at designing numerical schemes which allow step-sizes much larger than the cyclotron period. Upon assuming that the magnetic field has a fixed direction in space, i.e. that
 $$
 \bB(\bx)=(0,0,b(\bx))^T,\quad b(\bx)>0
 $$
(a popular choice both for formal and rigorous analyses \cite{bostan, degond, Golse}), several multiscale numerical methods have been proposed  \cite{vpb,UAVP4d,Yang,Filbet_rod,Filbet_rod2,Sonnendrucker2}.
Among them, Filbet et al. constructed Particle-In-Cell schemes in the spirit of asymptotic preserving techniques \cite{Shi}, which, as $\eps\to0$,  are consistent with the drift-limit model \cite{Filbet_rod} or the gyrokinetic model \cite{Yang,Filbet_rod2}. These schemes are simple and highly accurate, but the gyro-motion is lost in the limit regime. In contrast, the schemes proposed in \cite{vpb,UAVP4d} capture all the information of the kinetic models with an accuracy uniform with respect to $0<\eps\leq1$. These \emph{uniformly accurate (UA)} schemes have  computational cost as well as accuracy  totally independent of $\eps$ (we refer to \cite{CPC} for a comparison of UA scheme with other multiscale methods). In order to design UA schemes for kinetic models, different numerical approaches may be used: (i) The two-scale formulation technique relies upon an explicit separation of the fast and slow times and allows to smooth out the oscillations \cite{vpb,UAVP2d,UAVP4d}. (ii) The multi-revolution composition methods, in the spirit of heterogeneous multiscale method \cite{HMM}, are also  UA, as confirmed in the recent paper \cite{FLR}. Both approaches exploit the periodicity of the solution of the stiff part of the equation. For instance,  our recent work \cite{vpb} isolates the dominant oscillation frequency owing to a confining property in two dimensions. However, the general case of a strong  magnetic field with varying direction, i.e.
$$\bB(\bx)
=(b_1(\bx),b_2(\bx),b_3(\bx))^T,$$
has been barely considered so far for the Vlasov-Poisson equation (\ref{eq:1}) due to its complicated highly-oscillatory behaviour in three dimensions. Let us also mention recent developments around symplectic Particle-In-Cell method, which allows
for good preservation of invariants for very long time (see \cite{Lubich, qin, gempic}).

In this work, we propose efficient numerical schemes  for solving the three dimensional Vlasov-Poisson equation (\ref{eq:1}) in the strong magnetic field regime by combining multiscale strategies with the Particle-In-Cell (PIC) discretisation. First, we consider the case of a magnetic field with constant intensity $|\bB(\bx)|=const$, for which, as already pointed out in \cite{AAT,straymond}, the motion induced by the stiff Lorentz term $\frac{1}{\eps}\bv
\times\bB(\bx)$ in (\ref{eq:1}) is periodic in time. Taking advantage of this observation, we derive the limit model of (\ref{eq:1}) by using averaging methods \cite{AAT}, and then introduce three UA schemes, namely (i) the multi-revolution composition (MRC) method, (ii) the two-scale formulation (TSF) method and (iii) the micro-macro (MM) method. All three are of uniform second order in time for all $\eps\in]0,1]$, though have specific pros and cons: for instance, MRC methods are phase-space volume preserving, while MM easily allows for the full recovery of the gyro-motion. To the best of our knowledge, this key-feature is new and paves the way for an extension to the case of a magnetic field with varying intensity. In this situation, we indeed introduce, under the PIC discretisation, a reparametrization of time  to re-normalise the magnetic field. Within this framework, each particle carries its own fictitious time. Hence, and in order to avoid the occurrence of multiple frequencies, we drop in this situation the Poisson part of (\ref{eq:1}) and consider instead the case of an external electric field $\bE(t,\bx)$ (this somehow simplifying assumption is relevant as it marks an important first step towards the solution of the full problem). In order to re-synchronise all particles (a necessary step in order to provide an approximation of $f^\eps(t,\bx,\bv)$), we then use the interpolation strategy of MM which ensures uniform second order except for the angular variable. Eventually, numerical experiments  are presented in order to validate uniform accuracy and to compare  the various methods. In particular, we  simulate the dynamics of equation (\ref{eq:1}) in a three dimensional screw-pinch setup \cite{Michel}.

The remaining  of the paper is now organized as follows. Section \ref{sec:lm} considers the limit model of (\ref{eq:1}) with a constant intensity $\bB(\bx)$ and Section \ref{sect:nm} introduces the three aforementioned UA schemes in this situation: Subsection \ref{sect:mrc} is concerned with MRC method, Subsection \ref{sec:2scale} with TSF method and Subsection \ref{sec:mm} with MM method. Extensions to the case of a varying intensity are presented in Section \ref{sec:ex}. Finally, numerical results with concluding remarks are exposed in Section \ref{sect:nr}.

\section{Averaging}\label{sec:lm}
A general assumption throughout this paper is that  the magnetic field  is bounded from below, i.e. that $|\bB(\bx)|\geq c_0$ for all $\bx\in\bR^3$ for some $c_0>0$ independent of $\eps$. In this section, we further assume that the external magnetic field has constant norm
$$|\bB(\bx)|\equiv const>0,\quad \bx\in\bR^3,$$
so that the stiff part of equation (\ref{eq:1}) generates periodic motion.  This setup has also been considered  in \cite{Golse}.  A possible instance of such a $\bB$ is given by  $\bB(\bx)=(B_1(x_1,x_2),$ $B_2(x_1,x_2),B_3(x_1,x_2))$ where
$$B_3=\sqrt{\|B_1^2+B_2^2\|_{L^\infty(\bR^2)}-B_1^2-B_2^2},$$
with  $(B_1(x_1,x_2),B_2(x_1,x_2))\in L^\infty(\bR^2)$ and $\partial_{x_1}B_1+\partial_{x_2}B_2=0$. It can be verified that
$|\bB(\bx)|^2\equiv\|B_1^2+B_2^2\|_{L^\infty(\bR^2)}$ and $\nabla_x\cdot \bB\equiv0$. In such a case, we are able to apply a recently developed averaging method to quickly obtain the limit model of (\ref{eq:1}) as $\eps\to0$.

\begin{lemma}
If $|\bB(\bx)|\equiv b$ for some constant $b>0$, then the solution of
\begin{equation}\label{lm:1}\partial_t \tilde{f}^\eps(t,\bx, \bv) +\frac{1}{\eps}\bv\times \bB(\bx)\cdot\nabla_\bv \tilde{f}^\eps(t,\bx, \bv)=0,
\end{equation}
 is $2\pi/b$-periodic with respect to the  fast time-variable $t/\eps$.
\end{lemma}
\begin{proof}
The characteristics of (\ref{lm:1})
$$
\dot{\bx}(t)=0,\quad \dot{\bv}(t)=\frac{1}{\eps}\bv(t)\times\bB(\bx(t)),\quad t>0,$$
have a periodic solution in $t/\eps$ which can be obtained, for instance, by Rodrigues' formula
\begin{align}\label{rodrigues}
\bx(t)&=\bx(0), \nonumber \\
\bv(t)&=\cos(b t/\eps)\bv(0)+(1-\cos(b t/\eps))(\bB(\bx(0))\cdot\bv(0))\bB(\bx(0))
-\sin(b t/\eps)\bv(0)\times\bB(\bx(0)).
\end{align}
The statement of the lemma is now an immediate consequence.
\end{proof}
Using the observation above, we may apply the following theorem from \cite{AAT}:
\begin{theorem}\label{thm}
 Consider a  transport equation of form
  $$\partial_tf^\eps(t,\by)+\left[\frac{G(\by)}{\eps}+K(\by)\right]\cdot\nabla_\by f^\eps(t,\by)=0,\quad f^\eps(0,\by)=f_0(\by),$$
  where the flow map $\Phi_t$ of
  $$\dot{\by}(t)=G(\by(t))$$
  is assumed to be $2\pi$-periodic. There exist two {\em formal} vector fields $G^\eps(\by)$ and $K^\eps(\by)$ satisfying$$\frac{G(\by)}{\eps}+K(\by)=\frac{G^\eps(\by)}{\eps}+K^\eps(\by) \quad \mbox{ and } \quad [G^\eps, K^\eps]=0,$$
 such that the system
\begin{subequations} \label{eq:com}
\begin{numcases}
\,\partial_\tau g(t,\tau,\by)+G^\eps(\by)\cdot\nabla_\by g(t,\tau,\by)=0,\\
\partial_t g(t,\tau,\by)+K^\eps(\by)\cdot\nabla_\by g(t,\tau,\by)=0,\\
g(0,0,\by)=f_0(\by).
\end{numcases}
\end{subequations}
has a unique formal solution independently of the order in which the equations are solved. Moreover, for all positive time we have $f^\eps(t,\by)=g(t,t/\eps,\by)$ and the first two terms of $K^\eps=K^{[1]}+\eps K^{[2]}+O(\eps^2)$ may be computed as follows
  $$K^{[1]}=\Pi K_\tau,\quad K^{[2]}=-\frac{1}{2}\Pi\int_0^\tau[K_s,K_\tau] ds,\quad \mbox{ with } \quad
  K_\tau(\by):=\left(D_\by\Phi_\tau(\by)\right)^{-1}(K\circ\Phi_\tau)(\by), $$
 with $\Pi h := 1/(2\pi) \int_0^{2\pi} h(\tau) d\tau$.
\end{theorem}
\begin{remark}
If $K^\eps$ is truncated at order $k$ in $\eps$ and $G^\eps=G(y)+\eps K(y)-\eps K^\eps$, then the order in which the equations in (\ref{eq:com}) are solved does matter. However, the difference between the corresponding two solutions is also of order $\eps^k$.
\end{remark}
Without loss of generality, we assume in the rest of this section that $|\bB(\bx)|\equiv1$ to derive the averaged model of equation (\ref{eq:1})  as obtained from Theorem \ref{thm}. Details of the derivation may be found in \cite{AAT} and we thus content ourselves with a sketch of the computations:
from (\ref{eq:1}), we have
$$G=\begin{pmatrix}\mathbf{0}\\ \bv\times\bB\end{pmatrix},\quad K=\begin{pmatrix}\bv\\ \bE\end{pmatrix},$$
and the flow map generated by $G$ is given by
$$\Phi_\tau(\by)=\begin{pmatrix}\bx\\\cos(\tau)\bv+\sin(\tau)\bv\times \bB+(1-\cos(\tau))(\bB\cdot\bv)\bB \end{pmatrix},\quad
\by=\binom{\bx}{\bv}.$$
Then
$$D_\by\Phi_\tau(\by)=\begin{pmatrix}I_3&\mathbf{0}\\N_1&N_2 \end{pmatrix},$$
where
\begin{align*}
&N_1=\sin(\tau) \, (\bv\times \nabla_\bx \bB) +(1-\cos(\tau))[(\bB\cdot\bv)\nabla_\bx\bB+\bB(\bv^T \nabla_\bx \bB)],\\
&N_2=\cos(\tau)I_3-\sin(\tau)(\bB\times I_3)+(1-\cos(\tau))\bB\bB^T
\end{align*}
where we have denoted $\bv\times \nabla_\bx \bB=[\bv\times \partial_{x_1} \bB,\bv\times \partial_{x_2} \bB,\bv\times \partial_{x_3} \bB]$ and similarly for $\bv \times I_3$.
A straightforward calculation then leads to
$$\Pi\left((D_\by \Phi_\tau(\by))^{-1}(K\cdot\Phi_\tau)\right)
=\binom{(\bB\cdot\bv)\bB}{\bB_\bv}
$$
where
\begin{align*}
  \bB_\bv=&\bB\bB^T\left[\bE-\frac{1}{2}(\bv\times \nabla_\bx \bB)(\bv\times\bB)
  +M\bv-\frac{5}{2}M(\bB\cdot\bv)\bB\right]-\frac{1}{2}M\left[\bv-2(\bB\cdot\bv)\bB\right]\\
  &-\frac{1}{2}\bB\times I_3\left[M(\bv\times\bB)+(\bv \times \nabla_\bx \bB)(\bB\cdot\bv)\bB\right],
\end{align*}
with
$$M=(\bB\cdot\bv) \nabla_\bx \bB+\bB(\bv^T \nabla_\bx \bB).$$
Eventually,
$$K^\eps=\binom{(\bB\cdot\bv)\bB}{\bB_\bv}+O(\eps),$$
and  limit model at leading order is
$$\partial_tg(t,\tau,\bx,\bv)+(\bB\cdot\bv)\bB\cdot\nabla_\bx g(t,\tau,\bx,\bv)+\bB_\bv\cdot\nabla_\bv g(t,\tau,\bx,\bv)=0,\quad t>0.$$
By taking $\tau=0$ and $f(t,\bx,\bv)=g(t,0,\bx,\bv)$, we get the leading order averaged model of (\ref{eq:1})  for stroboscopic times $t \in 2\pi \eps \bN$,
\begin{subequations}\label{eq:lm}
\begin{align}
&\partial_t f(t,\bx, \bv) +(\bB\cdot\bv)\bB\cdot\nabla_\bx f(t,\bx, \bv)+\bB_\bv\cdot\nabla_\bv f(t,\bx, \bv)=0,\\
&\nabla_\bx\cdot \bE(t,\bx)=\int_{\bR^3}f(t,\bx,\bv)d\bv-n_i,\\
&f(0,\bx,\bv)=f_0(\bx,\bv).
\end{align}
\end{subequations}
As we shall verify later numerically, we have
$$f^\eps(t,\bx,\bv)-f(t,\bx,\bv)=O(\eps),\quad 0<\eps\ll1, \quad  t \in2\pi \eps \bN.
$$

\section{Numerical method} \label{sect:nm}
In this section, we introduce numerical schemes for equation (\ref{eq:1}) under the assumption $|\bB(\bx)|\equiv1$. Taking advantage of the periodicity the solution of the stiff part, we apply state-of-art multiscale approaches in combination with  PIC discretisation. In this way, we obtain
schemes whose accuracy and computational cost are both independent of  $\eps\in]0,1]$.
Our starting point is the following PIC-representation of $f^\eps$ as used in, e.g.,  \cite{Filbet_rod,Sonnendrucker2,Hesthaven,SonnendruckerBook},
\begin{equation}\label{dirac}
f^\eps(t,\bx,\bv)\approx\sum_{k=1}^{N_p}\omega_k\delta(\bx-\bx_k(t))\delta(\bv-\bv_k(t)),\quad t\geq0,\ \bx,\bv\in\bR^2.
\end{equation}
The characteristic equations of model (\ref{eq:1}) for $1\leq k\leq N_p$ are then of the form
\begin{subequations}
\label{charact}
\begin{align}
   & \dot{\bx}_k(t)=\bv_k(t), \\
  &\dot{\bv}_k(t)=\bE(t,\bx_k(t))+\frac{1}{\eps}\bv_k(t)\times \bB(\bx_k(t)),\quad t>0, \\
   & \bx_k(0)=\bx_{k,0},\quad \bv_k(0)=\bv_{k,0}.
\end{align}
\end{subequations}
Noticing that
$$\nabla_\bx\cdot\bE(t,\bx)=\sum_{k=1}^{N_p}w_k\delta(\bx-\bx_k(t)),$$
we observe that the electric field $\bE$ in (\ref{charact}) has in fact no explicit dependence on time, i.e. $\bE(t,\bx)=\bE_{[X(t)]}(\bx)$  where $X(t)=(\bx_1(t),\ldots,\bx_{N_p}(t))$. We are in a position to  briefly present three different  UA methods.

\subsection{Multi-revolution composition method}\label{sect:mrc}
For a general exposition of  multi-revolution composition (MRC), we refer to \cite{MRM}. Here, we focus on a uniformly accurate second-order method.

\textbf{MRC framework.}
Suppose that we wish to solve equation (\ref{eq:1}) on $[0,T_f]$ for some $T_f>0$. Rescaling time in  (\ref{charact}) leads to (we omit the particle index for brevity)
\begin{subequations} \label{MRC charact}
\begin{align}
   & \dot{\bx}(t)=\eps\bv(t),\\
  &\dot{\bv}(t)=\eps\bE_{[X(t)]}(\bx(t))+\bv(t)\times \bB(\bx(t)), \quad 0<t\leq\frac{T_f}{\eps},\\
   & \bx(0)=\bx_{0},\quad \bv(0)=\bv_{0},
\end{align}
\end{subequations}
Since the stiff part of (\ref{MRC charact}) generates a  $2\pi$-periodic motion, equation (\ref{MRC charact}) is amenable to MRC \cite{MRM,M2AN}.
To do so, we write
\begin{equation}
\label{def_Tf}
\frac{T_f}{\eps}=2\pi M_f+T_r,\quad M_f=\lfloor{\frac{T_f}{ 2\pi\eps}\rfloor}\in\bN, \quad 0\leq T_r<2\pi.
\end{equation}
The 2nd order MRC method begins by choosing an integer $0<M_0\leq M_f$ and defining
\begin{equation}\label{alphabeta}\alpha=\frac{1}{2}\left(1+\frac{1}{M_0}\right),\quad \beta=\frac{1}{2}\left(1-\frac{1}{M_0}\right),\quad M=\frac{M_f}{M_0},\quad H=\eps M_0.
\end{equation}
Denoting $\bx^n\approx\bx(2\pi nM_0),\,\bv^n\approx\bv(2\pi nM_0)$,  the MRC scheme proceeds as follows
\begin{equation}
\label{MRC main}
\binom{\bx^{n+1}}{\bv^{n+1}}=\mathcal{E}_\beta(-2\pi)\mathcal{E}_\alpha(2\pi)\binom{\bx^{n}}{\bv^{n}},\quad 0\leq n\leq M-1,
\end{equation}
where $\mathcal{E}_\alpha(2\pi)$ denotes the value at time $2\pi$ of the flow of
\begin{numcases}
   \, \dot{\bx}(t)=\alpha H\bv(t), \nonumber\\
  \dot{\bv}(t)=\alpha H\bE_{[X(t)]}(\bx(t))+\bv(t)\times \bB(\bx(t)),\label{MRC alpha}
  \end{numcases}
and $\mathcal{E}_\beta(-2\pi)$ the value at time $(-2\pi)$ of the flow of
\begin{numcases}
   \, \dot{\bx}(t)=-\beta H\bv(t), \nonumber\\
  \dot{\bv}(t)=-\beta H\bE_{[X(t)]}(\bx(t))+\bv(t)\times \bB(\bx(t)).\label{MRC beta}
\end{numcases}
The solution at final time $T_f$ is then obtained by applying to $\binom{\bx^{M}}{\bv^{M}}$ the flow $\mathcal{E}_r(T_r)$ at time $T_r$ of
\begin{numcases}
   \, \dot{\bx}(t)=\eps \bv(t), \nonumber\\
  \dot{\bv}(t)=\eps \bE_{[X(t)]}(\bx(t))+\bv(t)\times\bB(\bx(t)). \label{remain flow}
\end{numcases}

\textbf{Splitting scheme.}
The full MRC scheme calls for the numerical evaluation of the sub-flows $\mathcal{E}_\alpha(2\pi)$, $\mathcal{E}_\beta(2\pi)$ and $\mathcal{E}_r(T_r)$. This is done here through a splitting, for instance of $\mathcal{E}_\alpha$, in
\begin{equation}\label{split subflow}
\mathcal{E}_\alpha^{\bx}(t):\left\{\begin{split}
 &\dot{\bx}(s)=\alpha H\bv(s), \\
  &\dot{\bv}(s)=0,  \quad 0<s\leq t,
  \end{split}\right.\quad \mbox{and }\quad
  \mathcal{E}_\alpha^{\bv}(t):\left\{\begin{split}
 &\dot{\bx}(s)=0, \quad 0<s\leq t,\\
  &\dot{\bv}(s)=\alpha H\bE_{[X(s)]}(\bx(s))+\bv(s)\times \bB(\bx(s)).
  \end{split}\right.
  \end{equation}
 Note that both $\mathcal{E}_\alpha^{\bx}(t)$ and $\mathcal{E}_\alpha^{\bv}(t)$ can be exactly integrated. The exact flow of $\mathcal{E}_\alpha^{\bx}(t)$ is clearly
  $$\bx(t)=\bx(0)+t\alpha H\bv(0),\quad \bv(t)=\bv(0),\quad t\geq0,$$
while the exact flow of $\mathcal{E}_\alpha^{\bv}(t)$, by using the Rodrigues' rotation formula,  can also be written explicitly
  \begin{align*}
  \bx(t)=&\bx(0), \\
  \bv(t)=&\cos(t)\bv(0)+\sin(t)\bv(0)\times \bB+\alpha H\sin(t)\bE+\alpha H(t-\sin(t))(\bB\cdot
  \bE)\bB+\alpha H(1-\cos(t))\bE\times\bB\\
  &+(1-\cos(t))(\bB\cdot\bv(0))\bB,\quad t\geq0,
  \end{align*}
  where $\bE=\bE_{[X(0)]}(\bx(0))$ and $\bB=\bB(\bx(0))$. In our experiments, we shall take the value of the (micro) time step $h=2\pi/M$, so that
  $$\mathcal{E}_\alpha(2\pi)\approx
  \left(\mathcal{E}_\alpha^{\bx}(h/2)\mathcal{E}_\alpha^{\bv}(h)\mathcal{E}_\alpha^{\bx}(h/2)\right)^M.$$
Approximations for $\mathcal{E}_\beta(2\pi)$ and $\mathcal{E}_r(T_r)$ are obtained in a similar way. It may then be proved (see \eqref{MRC main})
that the error of MRC is of size  $O(M^{-2})$ for a computational cost of size $M^2$, making the overall scheme of order
one\footnote{Under the assumption that $\bE(t,\bx)\in C^2(\bR^+\times\bR^3;\bR^3)$ and $\bB(\bx)\in C^2(\bR^3;\bR^3)$ in (\ref{charact}).}.

It remains to comment on what happens when the user-controlled $M$ increases to the limit where $M_0$ reaches the critical value $M_0=1$, for which $\alpha=1,\,\beta=0$ (\ref{alphabeta}) and $\mathcal{E}_\beta(-2\pi) \equiv id$. In this  case, the full MRC scheme may be regarded as just the discretisation of  equation (\ref{MRC charact}) by Strang's method with time step $h$. Therefore, as soon as $M>0$ implies $M_0=T_f/\eps/(2\pi)/M<1$, we replace MRC method by  Strang splitting with time step $h=2\pi/M$. Finally, note that all vector fields involved in MRC are divergence free so that their exact flows are \textbf{phase-space volume preserving} as is the MRC method itself.


\subsection{Two-scale formulation method}\label{sec:2scale}
Two-scale formulation (TSF) methods have been developed in \cite{UAKG,APVP2d}. Their underlying rationale is to consider the fast time as an additional variable. In order to isolate the fast time, we apply the change of unknowns $(\bx(t),\bv(t)) \mapsto (\bx(t),\by(t))$ where
\begin{equation}\label{filter}
\by(t)=\cos(t/\eps)\bv(t)+(1-\cos(t/\eps))(\bB(\bx(t))\cdot\bv(t))\bB(\bx(t))
-\sin(t/\eps)\bv(t)\times\bB(\bx(t)).
\end{equation}
This leads to
\begin{numcases}
\, \dot{\bx}(t)=F_\bx(t/\eps,\bx(t),\by(t)),\nonumber\\
\dot{\by}(t)=F_\by(t/\eps,\bx(t),\by(t)),\quad t>0,\label{filter charact}\\
\bx(0)=\bx_0,\quad \by(0)=\bv_0, \nonumber
\end{numcases}
where
$$F_\bx(\tau,\bx,\by):=\cos(\tau)\by+(1-\cos(\tau))(\bB(\bx)\cdot\by)\bB(\bx)+\sin(\tau)\by\times\bB(\bx),
$$
and
\begin{align*}
F_\by(\tau,\bx,\by)=&\cos(\tau)\bE_{[X]}(\bx)+(1-\cos(\tau))(\bB(\bx)\cdot\bE_{[X]}(\bx))\bB(\bx)-\sin(\tau)\bE_{[X]}(\bx)\times\bB(\bx)\\
&-\frac{1}{2}\sin(2\tau)\bq_\tau(\by)-\frac{1}{2}(2\sin(\tau)-\sin(2\tau))\bq_\tau((\bB(\bx)\cdot\by)\bB(\bx))\\
&-\frac{1}{2}(1-\cos(2\tau))\bq_\tau(\by\times\bB(\bx))+\frac{1}{2}(2\cos(\tau)-\cos(2\tau)-1)\bp_\tau(\by)+\frac{1}{2}(3-4\cos(\tau)\\
&+\cos(2\tau))\bp_\tau((\bB(\bx)\cdot\by)\bB(\bx))+\frac{1}{2}(2\sin(\tau)-\sin(2\tau))\bp_\tau(\by\times\bB(\bx)),
\end{align*}
with the vector fields
\begin{align*}
&\bp_\tau(\bz):=((\nabla_\bx \bB(\bx)F_\bx(\tau,\bx,\by))\cdot\bz)\bB(\bx)+(\bB(\bx)\cdot\bz)(\nabla_\bx\bB(\bx)F_\bx(\tau,\bx,\by)),\\
&\bq_\tau(\bz):=\bz\times(\nabla_\bx \bB(\bx)F_\bx(\tau,\bx,\by)),\quad \bz\in\bR^3.
\end{align*}
Denoting $\bu(t)=\binom{\bx(t)}{\by(t)}$ and $F(\tau,\bu)=\binom{F_\bx(\tau,\bx,\by)}{F_\by(\tau,\bx,\by)}$, the two-scale formulation of  system (\ref{filter charact}) now reads
\begin{align}
&\partial_tU(t,\tau)+\frac{1}{\eps}\partial_\tau U(t,\tau)=F(\tau,U(t,\tau)),\quad t>0,\ \tau\in\bT,\label{two-scale}\\
&U(0,0)=\bu(0),\nonumber
\end{align}
where $\bT=[0,2\pi]$, and one recovers the solution of (\ref{filter charact}) by taking the diagonal, i.e.
$$U(t,t/\eps)=\bu(t),\quad t\geq0.$$
It remains to prescribe an appropriate initial data $U(0, \tau)$ to \eqref{two-scale} 
so that the solution $U$ has its derivatives uniformly bounded up to some order.  \\ \\
\textbf{Initial data.}
In order to derive $U(0,\tau)$, we follow the Chapman-Enskog procedure. From the decomposition
$$\underline{U}(t)=\Pi U(t,\cdot),\quad \bh(t,\tau)=U(t,\tau)-\underline{U}(t), \; \mbox{ with } \, \Pi U(t,\cdot) =\frac{1}{2\pi}\int_0^{2\pi} U(t, \tau)d\tau,$$
we split (\ref{two-scale}) into
\begin{numcases}
\, \dot{\underline{U}}(t)=\Pi F(\cdot,\underline{U}(t)+\bh(t,\cdot)),\quad t>0,\nonumber\\
\partial_t\bh(t,\tau)+\frac{1}{\eps}\partial_\tau\bh(t,\tau)=(I-\Pi)F(\tau,\underline{U}(t)+\bh(t,\tau)),\quad t>0,\ \tau\in\bT.
\nonumber
\end{numcases}
Denote $L=\partial_\tau,\,A=L^{-1}(I-\Pi)$ and we have
$$\bh(t,\tau)=\eps AF(\tau,\underline{U}(t)+\bh(t,\tau))-\eps L^{-1}\partial_t\bh(t,\tau).$$
Differentiate the above with respect to $t$ on both sides:
$$\partial_t\bh(t,\tau)=\eps A\nabla F(\tau,\underline{U}+\bh)(\dot{\underline{U}}+\partial_t\bh)-\eps L^{-1}\partial_t^2\bh(t,\tau).$$
By assuming that $\partial_t^2\bh=O(1)$ for $\eps\in]0,1]$, one gets $\partial_t\bh=O(\eps)$ and $\bh(t,\tau)$ has the first order asymptotic expansion:
$$\bh(t,\tau)=\eps AF(\tau,\underline{U}(t))+O(\eps^2).$$
Using the fact
$\underline{U}(0)=\bu(0)-\bh(0,0),$
one gets at initial time
$$\bh(0,\tau)=\bh^{1st}(\tau)+O(\eps^2),\quad \mbox{with}\quad \bh^{1st}(\tau):=\eps AF(\tau,\bu(0)),$$
and we denotes the first order  initial data as:
\begin{equation}\label{1st data}
U^{1st}(\tau):=\bu(0)+\bh^{1st}(\tau)-\bh^{1st}(0).
\end{equation}
In fact, one can show rigorously that the equation (\ref{two-scale}) with the well-prepared initial data $U(0,\tau)=U^{1st}(\tau)$ offers
\begin{equation}\label{bound}\partial_tU(t,\tau),\,\partial_t^2U(t,\tau)=O(1),\quad \eps\in]0,1].
\end{equation}
We refer the readers to \cite{UAKG} for the mathematical justification. The boundedness of the time derivatives (\ref{bound}) is the key to design UA schemes.

\textbf{Exponential integrator.}
Thanks to the two-scale formulation  (\ref{two-scale}) with the well-prepared initial data $U(0,\tau)=U^{1st}(\tau)$ from (\ref{1st data}), we can now directly apply the second order exponential integrator scheme proposed in \cite{UAVP4d} for integrating (\ref{two-scale}): Choose $N_\tau>0$ an even integer to uniformly discretize $\tau$ on $\bT$ and take a $\Delta t>0$ to define $t_n=n\Delta t$. Denote $U^n(\tau)\approx
U(t_n,\tau)$ for $n\geq0$ and let $U^0(\tau)=U(0,\tau)$. We update the $U^n$ for $n\geq1$ as
\begin{subequations}\label{EI}
\begin{align}
&\widehat{(U)}_l^1
=\fe^{-\frac{il\Delta t}{\eps}}\widehat{(U)}_l^0
+p_l\widehat{(F)}_l^0
+q_l\frac{1}{\Delta t}\left(\widehat{(F)}_l^{*,1}-\widehat{(F)}_l^0\right),\\
&\widehat{(U)}_l^{n+1}=\fe^{-\frac{il\Delta t}{\eps}}\widehat{(U)}_l^{n}
+p_l\widehat{(F)}_l^n
+q_l\frac{1}{\Delta t}\left(\widehat{(F)}_l^n-\widehat{(F)}_l^{n-1}\right),\quad n\geq1,
\end{align}
\end{subequations}
where for $n\geq0$,
$$
U^n(\tau)=\sum_{l=-N_\tau/2}^{N_\tau/2-1}\widehat{(U)}_l^n\fe^{il\tau},\quad
F^n(\tau)=\sum_{l=-N_\tau/2}^{N_\tau/2-1}\widehat{(F)}_l^n\fe^{il\tau},
\quad F^{*,1}(\tau)=\sum_{l=-N_\tau/2}^{N_\tau/2-1}\widehat{(F)}_l^{*,1}\fe^{il\tau},$$
and $F^n(\tau)=F(\tau,U^n(\tau)),\ F^{*,1}(\tau)=F(\tau,U^{*,1}(\tau))$
with
$$\widehat{(U)}_l^{*,1}
=\fe^{-\frac{il\Delta t}{\eps}}\widehat{(U)}_l^0
+p_l\widehat{(F)}_l^0,\quad U^{*,1}(\tau)=\sum_{l=-N_\tau/2}^{N_\tau/2-1}\widehat{(U)}_l^{*,1}\fe^{il\tau},$$
and
\begin{equation*}
p_l=\left\{\begin{split}
&\frac{i\eps}{l}\left(\fe^{-\frac{il\Delta t}{\eps}}-1\right),\quad l\neq0,\\
&\Delta t,\qquad\qquad\qquad\ \ \, l=0,\end{split}\right.
\qquad
q_l=
\left\{\begin{split}
&\frac{\eps}{l^2}\left(\eps-\eps\fe^{-\frac{il\Delta t}{\eps}}-il\Delta t\right),\quad l\neq0,\\
&\frac{\Delta t^2}{2},\qquad\qquad\qquad\qquad\quad\ \ l=0.\end{split}\right.
\end{equation*}
Suppose we have the numerical solution $U^n(\tau)=\binom{X^n(\tau)}{Y^n(\tau)}$ from the above scheme, then the numerical solution $\bx^n\approx\bx(t_n),\,\bv^n\approx\bv(t_n)$ of the original characteristics (\ref{charact}) reads:
\begin{align*}
&\bx^n=X^n(t_n/\eps),\quad n\geq1,\\
&\bv^n=\cos(t_n/\eps)Y^n(t_n/\eps)
+(1-\cos(t_n/\eps))(\bB(\bx^n)\cdot Y^n(t_n/\eps))\bB(\bx^n)
+\sin(t_n/\eps)Y^n(t_n/\eps)\times\bB(\bx^n).
\end{align*}

The derivation and convergence analysis of the above scheme can be found in \cite{UAVP4d}.
Since the filter (\ref{filter}) involves the magnetic field $\bB(\bx)$, the filtered system (\ref{filter charact}) which is less smooth than the original form (\ref{charact}), needs more regularity for optimal convergence of the algorithm. Assuming that $\bE(t,\bx)\in C^2(\bR^+\times\bR^3;\bR^3)$ and $\bB(\bx)\in C^3(\bR^3;\bR^3)$ in (\ref{charact}), the two-scale formulation (TSF) exponential integrator (\ref{EI}) gives uniform second order accuracy in terms of $\Delta t$ for all $\eps\in]0,1]$ and uniform spectral accuracy in terms of $N_\tau$ (due to periodicity):
$$O(\Delta t^2+N_\tau^{-m_0}).$$
 The total cost of the TSF method is
$O(\Delta t^{-1}N_\tau\log N_\tau)$.

\subsection{Micro-macro method}\label{sec:mm}
Now, we present the main new method of this work. It is based on the micro-macro (MM) decomposition that has been proposed very recently in \cite{NUA}. We shall for the first time consider this approach for the Vlasov-Poisson equation and propose a second order UA scheme. The same notations will be adopted from the previous subsection.

\textbf{MM decomposition.}
By the averaging theory \cite{Sanders}, it is known that for general oscillatory problem
\begin{equation}\label{ode}
\dot{\bu}(t)=F(t/\eps,\bu(t)),\quad t>0,
\end{equation}
 with $2\pi$-periodicity in $\tau$ of $F(\tau,\bu)$, the solution can be written as a composition
\begin{equation}\label{average}
\bu(t)=\Phi_{t/\eps}\circ\Psi_t\circ\Phi_0^{-1}(\bu(0)),
\end{equation}
where $\Phi_\tau(\bv)$ is a change of variable with $2\pi$-periodicity in $\tau$ for some $\bv$, and $\Psi_t(\bv)$ is the flow map of the autonomous equation with initial value $\bv$:
$$\dot{\Psi}_t(\bv)=F_0(\Psi_t(\bv)),\quad \Psi_0(\bv)=\bv,$$
for some field $F_0$.
Though (\ref{average}) is known to hold theoretically for some $\Phi_\tau$ and $\Psi_t$, the explicit formulas of $\Phi_\tau$ and $\Psi_t$ are not available. In fact, by plugging (\ref{average}) back to the equation (\ref{ode}),
the change of variable $\Phi_\tau$, the flow map $\Psi_t$ and the averaged field $F_0$ can be seen to  satisfy the relation
\begin{equation}
\frac{1}{\eps}\partial_\tau\Phi_\tau(\bv)+D_\bv\Phi_\tau(\bv)
F_0(\bv)=F(\tau,\Phi_\tau(\bv)),\label{constrant}
\end{equation}
and moreover by taking averaging with respect to $\tau\in[0,2\pi]$ on both sides of (\ref{constrant}), one can define $F_0$ with $\Phi_\tau$,
$$ F_0=(\Pi D_\bv\Phi_\tau)^{-1}\Pi F(\cdot,\Phi_\tau).$$
The above two equalities cannot completely determine $\Phi_\tau$ and $\Psi_t$, so the standard averaging method imposes an extra condition $\Pi\Phi_\tau=id$, which uniquely defines the change of variable $\Phi_\tau$ in an implicit way through (\ref{constrant}). In general, it is not possible to solve (\ref{constrant}) to find out the exact $\Phi_\tau$. However, we can define an approximated $\Phi_\tau$ through a $k$th-order iteration
\begin{equation}\label{iteration}\Phi_\tau^{[k+1]}=id+\eps A\left(F(\tau,\Phi_\tau^{[k]})-D_\bv\Phi_\tau^{[k]}F_0^{[k]}\right),\quad
k\in\bN,
\end{equation}
with initially
$$\Phi^{[0]}_\tau=id,\quad F_0^{[0]}=\Pi F,$$
which asymptotically gives
$$\Phi_\tau=\Phi_\tau^{[k]}+O(\eps^{k+1}).$$
As a compensation to the composition (\ref{average}) by using the approximated function $\Phi_\tau^{[k]}$, a defect $\bw^{[k]}$ needs to be introduced:
\begin{equation}\label{MM decomposition}
\bu(t)=\Phi_{t/\eps}^{[k]}\circ\Psi_t^{[k]}\circ(\Phi_0^{[k]})^{-1}(\bu(0))+\bw^{[k]}(t),
\end{equation}
to ensure that there are no asymptotical truncations made to the exact solution. The decomposition
(\ref{MM decomposition}) is referred as the micro-macro decomposition of the solution of (\ref{ode}).

As a matter of fact, the first order approximation $\Phi_\tau^{[1]}$ given by the micro-macro decomposition, i.e. $k=0$ in (\ref{iteration}), coincides with the first order Chapman-Enskog expansion that we introduced in the previous subsection:
$$\Phi_\tau^{[1]}(\br)=\br+\eps AF(\tau,\br)=:\Theta(\tau,\br).$$
Thus, (\ref{filter charact}) or (\ref{ode}) has the first order micro-macro decomposition:
\begin{equation}\label{MM}
\bu(t)=\Theta(t/\eps,\br(t))+\bw(t),\quad t\geq0,
\end{equation}
where the macro part representing the averaged equation reads
\begin{numcases}
\,\dot{\br}(t)=\Pi F\left(\cdot,\Theta(\cdot,\br(t))\right),\quad t>0,\label{macro}\\
\br(0)=\bu(0)-\eps AF(\tau,\bu(0))\vert_{\tau=0},\nonumber
\end{numcases}
and the  micro part representing the equation for the defect reads
\begin{numcases}
\,\dot{\bw}(t)=G( t/\eps,\br(t),\bw(t)),\label{micro}\quad t>0, \\
\bw(0)= \eps A\left[F(\tau,\bu(0))-F(\tau,\br(0))\right]\vert_{\tau=0},\nonumber
\end{numcases}
with
\begin{align*}
G(\tau,\br,\bw)
:=F\left(\tau,\Theta(\tau,\br)+\bw\right)-(I-\Pi)F(\tau,\br)
-\frac{d}{dt}\Theta(\tau,\br(t)).
\end{align*}
In the first order micro-macro decomposition, the macro part (\ref{macro}) is smooth containing no high-frequencies. As for the micro part, it can be shown that \cite{NUA}
$$\bw(t)=O(\eps^2),\quad \partial_t\bw(t)=O(\eps),\quad \partial_t^2\bw(t)=O(1).$$
Thanks to the reformulation (\ref{filter charact}), we are able to consider this micro-macro approach for the characteristics (\ref{charact}).

\textbf{An integration scheme.} Now based on the micro-macro decomposed systems (\ref{macro}) and (\ref{micro}), we are going to propose a second order integration for solving (\ref{ode}) which is a compact formulation of (\ref{filter charact}) with $\bu(t)=\binom{\bx(t)}{\by(t)}$
and $F(\tau, \bu)=\binom{F_\bx(\tau, \bx, \by)}{F_\by(\tau, \bx, \by)}$.
We solve the macro part (\ref{macro}) by a leap-frog finite difference scheme:
\begin{align*}
&\br^{n+1}=\br^{n-1}+2\Delta t\Pi F\left(\cdot,\Theta(\cdot,\br^n)\right),\quad n\geq1,\quad
\br^{1}=\br^0+\Delta t \Pi F\left(\cdot,\Theta(\cdot,\br^0)\right).
\end{align*}
For the micro part (\ref{micro}), we integrate the equation to have
\begin{align}\bw(t_{n+1})-\bw(t_n)&=\int_{t_n}^{t_{n+1}}G(t/\eps,\br(t),\bw(t))dt\nonumber\\
&=\int_{t_n}^{t_{n+1}}H(t/\eps,\br(t),\bw(t))dt-\Theta(t_{n+1}/\eps,\br(t_{n+1}))+
\Theta(t_{n}/\eps,\br(t_{n})),\label{micro appro}
\end{align}
where
$$
H(\tau,\br,\bw):=F\left(\tau,\Theta(\tau,\br)+\bw\right).
$$
Since $H(\tau,\br,\bw)$ is periodic in $\tau\in\bT$, so we have a Fourier expansion
$$H(\tau,\br,\bw)=\sum_{l\in\bZ}\widehat{H}_l(\br,\bw)\fe^{il\tau},$$
and the integration in (\ref{micro appro}) can be approximated as
 \begin{align*}
 &\int_{t_n}^{t_{n+1}}H(t/\eps,\br(t),\bw(t))dt
=\sum_{l\in\bZ}\int_{t_n}^{t_{n+1}}\widehat{H}_l(\br(t),\bw(t))\fe^{ilt/\eps}dt\\
&\hspace{-0.6cm}\approx \sum_{l\in\bZ}\int_{t_n}^{t_{n+1}}\left[\widehat{H}_l(\br(t_n),\bw(t_n))
+(t-t_n)\frac{d}{dt}\widehat{H}_l(\br(t_n),\bw(t_n))\right]\fe^{ilt/\eps}dt\\
&\hspace{-0.6cm}\approx\sum_{l\in\bZ}\int_{t_n}^{t_{n+1}}\left[\widehat{H}_l(\br(t_n),\bw(t_n))
+\frac{t-t_n}{\Delta t}\left(\widehat{H}_l(\br(t_n),\bw(t_n))-\widehat{H}_l(\br(t_{n-1}),\bw(t_{n-1}))\right)\right]
\fe^{ilt/\eps}dt.
\end{align*}
Therefore, for $n\geq1$,
 \begin{align*}\bw(t_{n+1})
\approx& \; \bw(t_n)+\sum_{l\in\bZ}\fe^{ilt_n/\eps}\left[\alpha_l\widehat{H}_l(\br(t_n),\bw(t_n))
+\frac{\beta_l}{\Delta t}\left(\widehat{H}_l(\br(t_n),\bw(t_n))-\widehat{H}_l(\br(t_{n-1}),\bw(t_{n-1}))\right)\right]\\
&-\Theta(t_{n+1}/\eps,\br(t_{n+1}))+
\Theta(t_{n}/\eps,\br(t_{n})),
\end{align*}
and as for $n=0$,
\begin{align*}\bw(t_{1})
\approx& \; \bw(0)+\sum_{l\in\bZ} \alpha_l\widehat{H}_l(\br(0),\bw(0))-\Theta(t_{1}/\eps,\br(t_{1}))+
\Theta(0,\br(0)),
\end{align*}
where \begin{equation*}\begin{split}&\alpha_l=\int_0^{\Delta t}\fe^{ilt/\eps}dt={\left\{\begin{split}
&\frac{i\eps}{l}\left(1-\fe^{\frac{il\Delta t}{\eps}}\right),\quad l\neq0,\\
&\Delta t,\qquad\qquad\qquad\,  l=0,\end{split}\right.},\\
&\beta_l=\int_0^{\Delta t}t\fe^{ilt/\eps}dt=
{\left\{\begin{split}
&\frac{\eps}{l^2}\left((\eps-il\Delta t)\fe^{\frac{il\Delta t}{\eps}}
-\eps\right),\quad l\neq0,\\
&\frac{\Delta t^2}{2},\qquad\qquad\qquad\qquad\quad\ \ l=0.\end{split}\right.}
\end{split}\end{equation*}
In total, the detailed exponential integration scheme based on the micro-macro method reads:
\begin{subequations}\label{MM scheme}
\begin{align}
\bu^{n+1}=&\Theta(t_{n+1}/\eps,\br^{n+1})+\bw^{n+1},\quad n\geq0,\\
\br^{n+1}=&\br^{n-1}+2\Delta t\Pi F\left(\cdot,\Theta(\cdot,\br^n)\right),\quad n\geq1,\\
\bw^{n+1}
=&\bw^n+\sum_{l=-N_\tau/2}^{N_\tau/2-1}\fe^{ilt_n/\eps}\left[\alpha_l\widehat{H}_l(\br^n,\bw^n)
+\frac{\beta_l}{\Delta t}\left(\widehat{H}_l(\br^n,\bw^n)-\widehat{H}_l(\br^{n-1},\bw^{n-1})\right)\right]\\
&-\Theta(t_{n+1}/\eps,\br^{n+1})+
\Theta(t_{n}/\eps,\br^n),\quad n\geq1,\nonumber\\
\br^{1}=&\br^0+\Delta t \Pi F\left(\cdot,\Theta(\cdot,\br^0)\right),\quad
\bw^1=\bw^0+\sum_{l=-N_\tau/2}^{N_\tau/2-1} \alpha_l\widehat{H}_l(\br^0,\bw^n)-\Theta(t_{1}/\eps,\br^1)+
\Theta(0,\br^0),\\
\br^0=&\bu(0)-\eps AF(\tau,\bu(0))\vert_{\tau=0},\quad
\bw^0=\eps A\left[F(\tau,\bu(0))-F(\tau,\br(0))\right]\vert_{\tau=0},
\end{align}
\end{subequations}
where $N_\tau$ is an even integer to truncate the Fourier series. Suppose the numerical solution of MM is obtained as
$\bu^n=\binom{\bx^n}{\by^n}$, then the numerical velocity of (\ref{charact}) at $t_n$ is given as
\begin{align*}
&\bv^n=\cos(t_n/\eps)\by^n
+(1-\cos(t_n/\eps))(\bB(\bx^n)\cdot \by^n)\bB(\bx^n)
+\sin(t_n/\eps)\by^n\times\bB(\bx^n).
\end{align*}

The micro-macro (MM) scheme (\ref{MM scheme}) is uniformly second order accurate. In practical programming, one only needs a subroutine to evaluate $F(\tau,\bu)$. When the electric and magnetic field $\bE$ and $\bB$ in particle system (\ref{charact}) are given external functions such as polynomials, the dependence of the fast time scale $t/\eps$ (or $\tau$) in $F$ and $\Theta$ can be found out explicitly and the averaging with respect to $\tau$
(through the operator $\Pi$) in the MM scheme can be pre-computed exactly.
Then the MM method will have a discretisation error in time of $O(\Delta t^2)$ with
optimal computational cost $O(\Delta t^{-1})$. In case that the exact evaluation of $t/\eps$ is impossible or too costly, one can always perform those computations of the fast time scale with the additional variable $\tau$ by FFT with uniform spectral accuracy thanks to the periodicity. In such case, the error bound of MM is
$$O(\Delta t^2+N_\tau^{-m_0}),$$
and the total cost is $O(\Delta t^{-1}N_\tau\log N_\tau)$, which are the same as TSF.

\textbf{Full recovery of oscillation.}
Since MM method finds out the dependence of the fast scale in a rather explicit way, it can easily recover the complete gyro-motion of the particles, i.e. the full oscillatory trajectory of the solution of (\ref{charact}), by interpolating respectively the macro part and micro part.

Let $\br^{n}$ and $\bw^n$ be the numerical solutions obtained from MM under a step size $\Delta t>0$. For an arbitrary $t>0$, if $t_n<t<t_{n+1}$, then we can use the linear interpolation to get
$$\br_I^n(t)=\frac{t_{n+1}-t}{\Delta t}\br^n+\frac{t-t_{n}}{\Delta t}\br^{n+1},\quad \bw_I^n(t)=\frac{t_{n+1}-t}{\Delta t}\bw^n+\frac{t-t_{n}}{\Delta t}\bw^{n+1}.$$
Noting that $\br(t)$ is the averaged part and $\bw(t)$ satisfies $\partial_t^2\bw(t)=O(1)$, together with the accuracy order of $\br^{n}$ and $\bw^n$ from the MM scheme, it is clear that the above linear interpolation gives uniform second accuracy for approximating $\br(t)$ and $\bw(t)$.
Then with micro-macro decomposition (\ref{MM}), we get the interpolated numerical solution of (\ref{ode}) as
\begin{equation}\label{full recover}
\bu_I^{n}(t)=\Theta(t/\eps,\br_I^{n}(t))+\bw_I^{n}(t),\quad t_n\leq t\leq t_{n+1},
\end{equation}
which fully recovers the oscillation information with ease. It is direct to see
$$|\bu(t)-\bu_I^{n}(t)|=O(\Delta t^2+N_\tau^{-m_0}). $$

\textbf{Restart strategy.}
In practical long time computing, we observe that the MM scheme (\ref{MM scheme}) could have numerical instability issue. The instability is developed from the micro part (\ref{micro}) in MM decomposition (\ref{MM}) as time evolves, since $\bw(t)=O(\eps)$ does not hold for arbitrary long time in general. Here, we propose a restart strategy to improve its long time performance.

Choose $T_0>0$ as the period to restart the MM decomposition. For some $m\in\bN$, we consider the oscillatory problem (\ref{ode}) for $\bu^m(t)=\bu(mT_0+t)$ as
$$\dot{\bu}^{m}(t)=F(mT_0/\eps+t/\eps,\bu^m(t)),\quad 0<t\leq T_0.$$
Then we apply the proposed MM strategy on the above, which leads to MM decomposition as
\begin{equation}\label{MM re}
\bu^m(t)=\Theta(mT_0/\eps+t/\eps,\br(t))+\bw(t),\quad 0\leq t\leq T_0,
\end{equation}
with
\begin{numcases}
\,\dot{\br}(t)=\Pi F\left(mT_0/\eps+\cdot,\Theta(mT_0/\eps+\cdot,\br(t))\right),\quad 0< t\leq T_0,\nonumber\\
\br(0)=\bu^m(0)-\eps AF(mT_0/\eps+\tau,\bu^m(0))\vert_{\tau=0},\nonumber
\end{numcases}
and
\begin{numcases}
\,\dot{\bw}(t)=G(mT_0/\eps+ t/\eps,\br(t),\bw(t)),\quad 0< t\leq T_0,\nonumber \\
\bw(0)= \eps A\left[F(mT_0/\eps+\tau,\bu^m(0))-F(mT_0/\eps+\tau,\br(0))\right]\vert_{\tau=0}.\nonumber
\end{numcases}
The integration scheme (\ref{MM scheme}) is then applied to solve the above two systems.

As can be seen in the numerical results later, this restart strategy for solving (\ref{charact}) is stable in long time computing. Its accuracy and computational cost are essentially the same as the direct scheme without restart.

\begin{remark}
  In the case that
   $\bB(\bx)=\bB_0(\bx)+O(\eps)$ with $|\bB_0(\bx)|\equiv const$,
   all the proposed algorithms in this section can be extended to such case without any essential difficulties.
\end{remark}

\begin{remark}
  Although physically the magnetic field should be divergence free, i.e. $\nabla_\bx\cdot\bB=0$,  all the algorithms we proposed in this section do not rely on the divergence free property of $\bB(\bx)$ to offer the uniform accuracy.
\end{remark}

\section{Extension to varying intensity magnetic field}\label{sec:ex}
In this section, we extend previous methods to the case of Vlasov equation with a general magnetic field (whose intensity may vary), i.e.
$$|\bB(\bx)|=b(\bx)\neq const,\quad \bx\in\bR^3. $$
We start by commenting on the difficulties to be encountered in this situation.  As soon as $|\bB(\bx)|=b(\bx)$ varies with $\bx$ while remaining bounded from below by some $c_0$ independent of $\eps$, the characteristic equation for each particle
\begin{align*}
   & \dot{\bx}_k(t)=\bv_k(t), \\
  &\dot{\bv}_k(t)=\bE(t,\bx_k(t))+\frac{1}{\eps}\bv_k(t)\times \bB(\bx_k(t)),\quad t>0,
   \end{align*}
generate  high oscillations. However, the dynamics of the linear part of the equation is non-periodic and thus does not allow for the application of averaging techniques. A possible remedy consists in time-reparametrisation. However, another difficulty then arises from the Poisson equation itself
  $$\nabla_\bx\cdot\bE(t,\bx)=\sum_{k=1}^{N_p}w_k\delta(\bx-\bx_k(t)),$$
 which couples  a huge number ($N_p\gg1$) of particles with different frequencies.
 \vskip1ex
 {\textbf{Rescaling the  time for each particle}.}
  Each particle has a periodic oscillation with respect to its own time $s_k=s_k(t)$, given by
\begin{equation}
\label{new time}\dot{s}_k(t)=b(\bx_k(t)),\quad s_k(0)=0.
\end{equation}
Note that $s_k(t)$ is strictly increasing and that
$$s_k(t)\to\infty,\quad \mbox{as}\quad t\to\infty,$$
 since $b(\bx) \geq c_0 >0$.
Denoting $\tilde{\bx}_k(s_k):=\bx_k(t),\  \tilde{\bv}_k(s_k):=\bv_k(t)$, we indeed have
\begin{numcases}
  \,\frac{\mathrm{d}}{\mathrm{d}s_k}\tilde{\bx}_k(s_k)=\frac{\tilde{\bv}_k(s_k)}{b(\tilde{\bx}_k(s_k))},\nonumber\\
  \frac{\mathrm{d}}{\mathrm{d}s_k}\tilde{\bv}_k(s_k)=\frac{\bE(t(s_k),\tilde{\bx}_k(s_k))}{b(\tilde{\bx}_k(s_k))}
  +\frac{1}{\eps}\tilde{\bv}_k(s_k)\times\frac{\bB(\tilde{\bx}_k(s_k))}{b(\tilde{\bx}_k(s_k))},
  \quad s_k>0,\label{s}\\
  \tilde{\bx}_k(0)=\bx_{k,0},\quad \tilde{\bv}_k(0)=\bv_{k,0},\nonumber
  \end{numcases}
  where the intensity of the magnetic field is scaled to one. Assuming that the electric field $\bE(t,\bx)$ is a given external field with no $\eps$-dependent oscillation in $t$, then the particle system (\ref{s}) is decoupled for each $k$.
  Therefore, the numerical methods introduced in the previous section can all be applied to  (\ref{s}) for each particle in its own  time $s_k$ with uniform accuracy.

  In order to build up an approximation of the  function $f^\eps(t,\bx,\bv)$ through (\ref{dirac}), it is then necessary to re-synchronise  for all particles. However, reverting $s_k$ to the physical time $t$ is not straightforward, as it requires to numerically solve the nonlinear equation
  (\ref{new time}) or its equivalent for the inverse map
  \begin{equation}\label{inversion}
  \dot{t}(s_k)=1/b(\tilde{\bx}_k(s_k)),\quad t(0)=0.
  \end{equation}
Given a physical time $t=T>0$ (or conversely $S_k>0$), the best we can hope for is to determine $s_k(T)$ (or conversely $t(S_k)$) up to an error of size $O(\Delta t^p)$ if a $p$th-order numerical method is applied. This source of possible error  needs to be properly controlled. Here we illustrate how it can be done for the micro-macro (MM) method.

{\textbf{Interpolating to synchronise}.} For the sake of brevity, we omit $k$ and denote
$$\tilde{\bB}(\tilde{\bx})=\frac{\bB(\tilde{\bx})}{b(\tilde{\bx})},\quad
\tilde{\bE}(t,\tilde{\bx})=\frac{\bE(t,\tilde{\bx})}{b(\tilde{\bx})}.
$$
We filter (\ref{s}) and (\ref{inversion}) as before
by introducing
\begin{equation}\label{filters}
\tilde{\by}(s):=\cos(s/\eps)\tilde{\bv}(s)+(1-\cos(s/\eps))(\tilde{\bB}(\tilde{\bx}(s))
\cdot\tilde{\bv}(s))
\tilde{\bB}(\tilde{\bx}(s))-\sin(s/\eps)\tilde{\bv}(s)\times\tilde{\bB}(\tilde{\bx}(s)),
\end{equation}
 and obtain
 \begin{numcases}
  \,\frac{\mathrm{d}}{\mathrm{d}s}\tilde{\bx}(s)=\tilde{F}_\bx(s/\eps,\tilde{\bx}(s),
  \tilde{\by}(s)),\nonumber\\
  \frac{\mathrm{d}}{\mathrm{d}s}\tilde{\by}(s)=\tilde{F}_\by(s/\eps,\tilde{\bx}(s),
  \tilde{\by}(s)),
  \quad s>0,\label{timere}\\
  \frac{\mathrm{d}}{\mathrm{d}s}t(s)=\frac{1}{b(\tilde{\bx}(s))},\nonumber\\
  \tilde{\bx}(0)=\bx_{0},\quad \tilde{\by}(0)=\bv_{0},\quad t(0)=0,\nonumber
  \end{numcases}
which  has now the appropriate format  (\ref{ode}). Here $\tilde{F}_\by$ is defined similarly as in \eqref{filter charact} (see Section \ref{sec:2scale})  with the scaled vector fields $\tilde{\bE}$, $\tilde{\bB}$ and $\tilde{F}_\bx=F_\bx/b(\tilde{\bx})$. We then solve system (\ref{timere}) with the MM scheme (\ref{MM scheme}) with time step $\Delta s>0$ and denote $t^n \approx t(s_n)$ the numerical solution of $t(s_n)$ at $s_n=n\Delta s$.
Then, using the notations $\br:=(\br_\bx, \br_\by)$ and $\bw=(\bw_\bx, \bw_\by)$
for the macro and micro parts (see Section \ref{sec:mm}), the numerical solution of $\tilde{\bx}$ and $\tilde{\by}$ at $s_n=n\Delta s$
is 
 $$
 \tilde{\bx}^n:=\br_{\bx}^n+\eps A\tilde{F}_{\bx}(s_n/\eps,\br_{\bx}^n,\br_{\by}^n)+\bw_{\bx}^n\approx\tilde{\bx}(s_n),\quad
 \tilde{\by}^n:=\br_{\by}^n
 +\eps A\tilde{F}_{\by}(s_n/\eps,\br_{\bx}^n,\br_{\by}^n)+\bw_{\by}^n\approx
 \tilde{\by}(s_n).$$
 Assume that $b(\cdot)\in C^1(\bR^3)$ and $0<c_0\leq b(\bx)\leq C_b$ for all $\bx\in\bR^3$ for some $C_b>0$. Then, from
 \begin{align*}
  t^{n}&\geq t(s_n)-|t^n-t(s_n)|\geq \frac{s_n}{C_b}-C\Delta s^2
 \end{align*}
we see that whenever $\Delta s>0$  is small enough, the value of $t^n$ will eventually become greater than any arbitrary positive value.
For a given  final time $T>0$, we thus stop the algorithm when
$t^n\leq T\leq t^{n+1}$. Note that the function $t(s)$ satisfies
$$
\frac{\mathrm{d}^2}{\mathrm{d}s^2}t(s)=O(1),\quad 0<\eps\leq1,
$$
so that we can interpolate the value of $t(s)$ from $t^n$ and $t^{n+1}$ with second order uniform accuracy:
$$\theta:=\frac{T-t^{n+1}}{t^n-t^{n+1}},\quad T=\theta t^n+(1-\theta)t^{n+1},\quad
s^*=\theta s^n+(1-\theta)s^{n+1}.$$
Interpolation is further used to obtain
\begin{align*}
&
\br_{\bx}^*=\theta \br_{\bx}^n+(1-\theta)\br_{\bx}^{n+1},\quad \br_{\by}^*=\theta \br_{\by}^n+(1-\theta)\br_{\by}^{n+1},\\
&\bw_{\bx}^*=\theta \bw_{\bx}^n+(1-\theta)\bw_{\bx}^{n+1},\quad \bw_{\by}^*=\theta \bw_{\by}^n+(1-\theta)\bw_{\by}^{n+1}.
\end{align*}
As stated in Section \ref{sec:mm}, all functions used above in the interpolation  have uniformly bounded second order derivative. As a consequence, the so-obtained approximations are uniformly second order. Eventually, the  numerical solutions of (\ref{timere}) at $s=s(T)$ are given by
$$\tilde{\bx}(s(T))\approx \tilde{\bx}^*,\quad \tilde{\by}(s(T))\approx \tilde{\by}^*,$$
with
\begin{equation}\label{inter}\tilde{\bx}^*:=\br_{\bx}^*
 +\eps A\tilde{F}_{\bx}(s^*/\eps,\br_{\bx}^*,\br_{\by}^*)+\bw_{\bx}^*,\quad
 \tilde{\by}^*:=\br_{\by}^*
 +\eps A\tilde{F}_{\by}(s^*/\eps,\br_{\bx}^*,\br_{\by}^*)+\bw_{\by}^*.
 \end{equation}
 Note that the dependence in the  fast-time $s/\eps$ within the MM method only appears in the $O(\eps)$-terms. As a consequence, the approximation errors of $\tilde{\bx}(s(T))$ and $\tilde{\by}(s(T))$ by (\ref{inter}) are still of uniform second order, although an error is introduced on $s^*/\eps$ owing to $|s(T)-s^*|/\eps=O(\Delta s^2/\eps)$.

To reconstruct an approximation of the distribution function $f^\eps(T,\bx,\bv)$, we need $\bx(T)$ and $\bv(T)$. For the position variable, we directly have $\bx(T)=\tilde{\bx}(s(T))$ due to the definition. As for the velocity variable $\bv(T)$, we need to invert the change of variable (\ref{filters}), where the fast scale $s/\eps$ occurs in some $O(1)$-terms. However, the parallel component
$\bv_{\parallel}:=(\bB\cdot \bv )\bB/\|\bB\|^2$ of the velocity  as well as
$|\bv|$ do not suffer from the same problem, thanks to the following observations
(let us recall that $\| \tilde{\bB}(\tilde{\bx}(s)) \|^2=1$)
$$
\tilde{\bv}_\parallel(s)=(\tilde{\bB}(\tilde{\bx}(s))\cdot\tilde{\by}(s))\tilde{\bB}(\tilde{\bx}(s)),\quad |\tilde{\bv}(s)|=|\tilde{\by}(s)|,
$$
which allow to get
$$\bv_\parallel(T)\approx(\tilde{\bB}(\tilde{\bx}(s(T)))\cdot\tilde{\by}(s(T)))
\tilde{\bB}(\tilde{\bx}(s(T))),\quad |\bv(T)|=|\tilde{\by}(s(T))|.$$
Therefore, the strategy proposed in this section is of overall uniform second order for the computation of
$$\bx(t),\quad \bv_\parallel(t),\quad  |\bv(t)|,\qquad t\geq0.$$
This, in turn, allows for a uniformly accurate approximation of macroscopic quantities such as the density or the kinetic energy
$$\rho^\eps(t,\bx):=\int_{\bR^3}f^\eps(t,\bx,\bv)d\bv,\quad \rho^\eps_\bv(t,\bx):=\int_{\bR^3}|\bv|^2f^\eps(t,\bx,\bv)d\bv,$$
as well as the magnetic moment \cite{ Lubich,Northrop}
\begin{equation}\label{magnetic moment PDE}
\mu^\eps(t):=\int_{\bR^3}\int_{\bR^3}f^\eps(t,\bx,\bv)\frac{|\bv_\bot|^2}
{|\bB(\bx)|}d\bx d\bv,
\end{equation}
with $\bv_\bot:=\bv - \bv_\parallel$.


 \section{Numerical results} \label{sect:nr}

This section is devoted to present the numerical results from the proposed numerical schemes.
We shall firstly test and compare the accuracy, efficiency and long time performance of the schemes
considering a single test particle for some three dimensional simulations in the two following cases:
constant intensity and varying intensity magnetic field. Then, we shall focus on the nonlinear Vlasov-Poisson
case under the influence of a constant intensity magnetic field.

\subsection{Accuracy study}
We investigate the performance of the proposed numerical methods by considering
a single particle system in three dimensions:
\begin{align}
   & \dot{\bx}(t)=\bv(t), \nonumber\\
  &\dot{\bv}(t)=\bE(\bx(t))+\frac{1}{\eps}\bv(t)\times \bB(\bx(t)),\quad t>0,\label{particle sysm} \\
   & \bx(0)=\bx_{0},\quad \bv(0)=\bv_{0}.\nonumber
\end{align}
We take two 3D vector fields $\bB(\bx):\bx\in\bR^3\to\bR^3$ and $\bE(\bx)=-\nabla_\bx\phi(\bx)$ with some $\phi(\bx):\bR^3\to\bR$, then (\ref{particle sysm}) is an Hamiltonian system with
the energy conserved as
\begin{equation}\label{eneryg par}
{\mathcal H}_s(t):=\frac{1}{2}|\bv(t)|^2+\phi(\bx(t))= {\mathcal H}_s(0),\qquad t\geq0.
\end{equation}
Note here we do not require $\bB$ being divergence free for these accuracy tests, since all the presented properties of the proposed schemes hold in general as long as $\bB$ is a smooth enough vector field. Hence, we first focus on
a constant intensity magnetic field before considering the general case to test the proposed methods MRC, TSF and MM.

\begin{example}\label{example1}\emph{(Constant intensity). }
We take the two external fields in the system \eqref{particle sysm}
$$\bE(\bx)=\left(\begin{split}
&\cos(x_1/2)\sin(x_2)\sin(x_3)/2\\
&\sin(x_1/2)\cos(x_2)\sin(x_3)\\
&\sin(x_1/2)\sin(x_2)\cos(x_3)
\end{split}\right),\quad
\bB(\bx)=\left(\begin{split}
&\quad\sin(x_1+x_2)\\
&\cos(x_1+x_2)\sin(x_3)\\
&\cos(x_1+x_2)\cos(x_3)
\end{split}\right),\quad \bx=(x_1,x_2,x_3),$$
where $|\bB(\bx)|=1$ and $\bE(\bx)$ derives from the potential $\phi(\bx)$:
$$
\bE(\bx)=-\nabla_\bx\phi(\bx),\quad \phi(\bx)=-\sin(x_1/2)\sin(x_2)\sin(x_3).
$$
We choose the initial data for  (\ref{example1})  as
$$\bx_0=(1/3,-1/2,\sqrt{\pi}/2),\quad \bv_0=(1/2,\fe/4,-1/3).$$
A reference solution is obtained by using the fourth order Runge-Kutta method with small step size $\Delta t=10^{-5}$.

We firstly study the convergence of the three proposed methods (MRC, TSF and MM)
aiming to illustrate their uniform accuracy for all $\eps\in]0,1]$.
To do so, we solve the system under different $\eps$ till $T=\pi/2$ and compute the error
\begin{equation}
\label{def_err}
error=\frac{|\bx(T)-\bx^{num}|}{|\bx(T)|}+\frac{|\bv(T)-\bv^{num}|}{|\bv(T)|},
\end{equation}
where $\bx^{num}$ and $\bv^{num}$ are the numerical values obtained by the different schemes.
For the TSF and MM methods, we define the time step $\Delta t=T/M$ with $M\in\mathbb{N}^\star$
and we fix the grid points for $\tau$-direction as $N_\tau=32$. For MRC, we define the numerical parameters
from a given $M\in\mathbb{N}^\star$ as follows:
$H=\eps M_f/M$ and $h=2\pi/M$ if $M_f/M\geq1$ and $\Delta t=2\pi/M$ if $M_f/M<1$
($M_f$ being defined by \eqref{def_Tf}).

The error (defined by \eqref{def_err}) produced by the three methods at $T=\pi/2$
with respect to the number of (macro) grid points $M$ or with respect to $\eps$ is given in Figure \ref{fig:example1}.
As expected, the three methods enjoy the uniform second order accuracy property
since the rate of convergence is essentially insensitive to the value $\eps\in]0,1]$.
The typical behavior of uniformly accurate methods can be observed on the error as a function of $\eps$:
the curves obtained for different $M$ are almost parallel.
Note that the results obtained by TSF and MM are very close whereas the error produced by MRC becomes smaller when $\eps$ decreases.

In Figure \ref{fig:tau} we look at the error of TSF and MM with respect to the number of grid points $N_\tau$
in the auxiliary variable $\tau$ (the time step is fixed to $\Delta t=10^{-5}$).
This error is important to study since these two methods involve an additional variable $\tau$
which may make them less competitive. We can see  in Figure \ref{fig:tau} that the error decreases spectrally
as the number of grid points $N_\tau$ increases. Moreover, for small values of $\eps$, a very small number of $N_\tau$
is needed to reach high accuracy: $\eps\leq2^{-7}$, $N_\tau=16$ is enough for machine precision. Finally, let us remark
that the results obtained for MM is much less sensitive than TSF: when $\eps=1/2$, $N_\tau=32$ enables to reach machine precision for MM whereas TSF requires $N_\tau=128$.

We now intend to compare the efficiency of TSF, MM and MRC in different regimes ($\eps=1/2$ and $1/2^{14}$).
Let us first fix the numerical parameters. According to the previous comments, in the regime $\eps=1/2$
we take $N_\tau=128$  for TSF and $N_\tau=32$ for MM whereas
in the regime $\eps=1/2^{14}$, we take $N_\tau=8$ for both TSF and MM.
We test the long time behavior of the three methods by investigating the relative error on the numerical total energy defined by
\begin{equation}\label{energy error}
error(t^n)=\frac{|{\mathcal H}_s^n-{\mathcal H}_s(0)|}{|{\mathcal H}_s(0)|},\quad {\mathcal H}_s^n=\frac{1}{2}|\bv^n|^2+\phi(\bx^n),
\end{equation}
where ${\mathcal H}_s^n$ is the numerical approximation of ${\mathcal H}_s(t_n)$ given by \eqref{eneryg par}.
We plot in Figure \ref{fig:com} the error (considering the maximum of \eqref{energy error} among all the iterations)
against the computational time of the three methods for $\eps=1/2$ or $1/2^{14}$ (different time steps have been chosen).
For a given error, when $\eps=1/2$ the MRC method is more efficient than TSF or MM, but it is no longer true when $\eps$
becomes smaller. This is explained by the fact that $N_\tau$ can be chosen smaller in the asymptotic regime, making
TSF and MM more competitive. MRC for $\eps=1/2$ reads as the Strang splitting, while for $\eps=1/2^{14}$ the convergence of MRC becomes first order in terms of total computational cost.
Then, in Figure \ref{fig:energy}, the time history of \eqref{energy error} is plotted for the three methods till $T=32\pi$, for $\eps=1/2^{14}$.
The TSF and MM methods run with $N_\tau=32$ and $M=1024$ ($\Delta t=0.098$) or $M=2048$ ($\Delta t=0.049$).
We report that, for this test, MM becomes unstable in large time so that the restart strategy is used every $T_0=8\pi$.
For MRC, we used $M=64$ or $M=128$.
Figure \ref{fig:energy} clearly shows that MRC has the best long time behavior among the three methods. Indeed, TSF and MM
has a linear drift in the energy error as time evolves whereas for MRC, it remains of the same order (about $10^{-5}$) for large time.
Let us remark that the energy error converges quadratically for the three methods with respect to number of time grid points $M$.

 \begin{figure}[t!]
$$\begin{array}{cc}
\psfig{figure=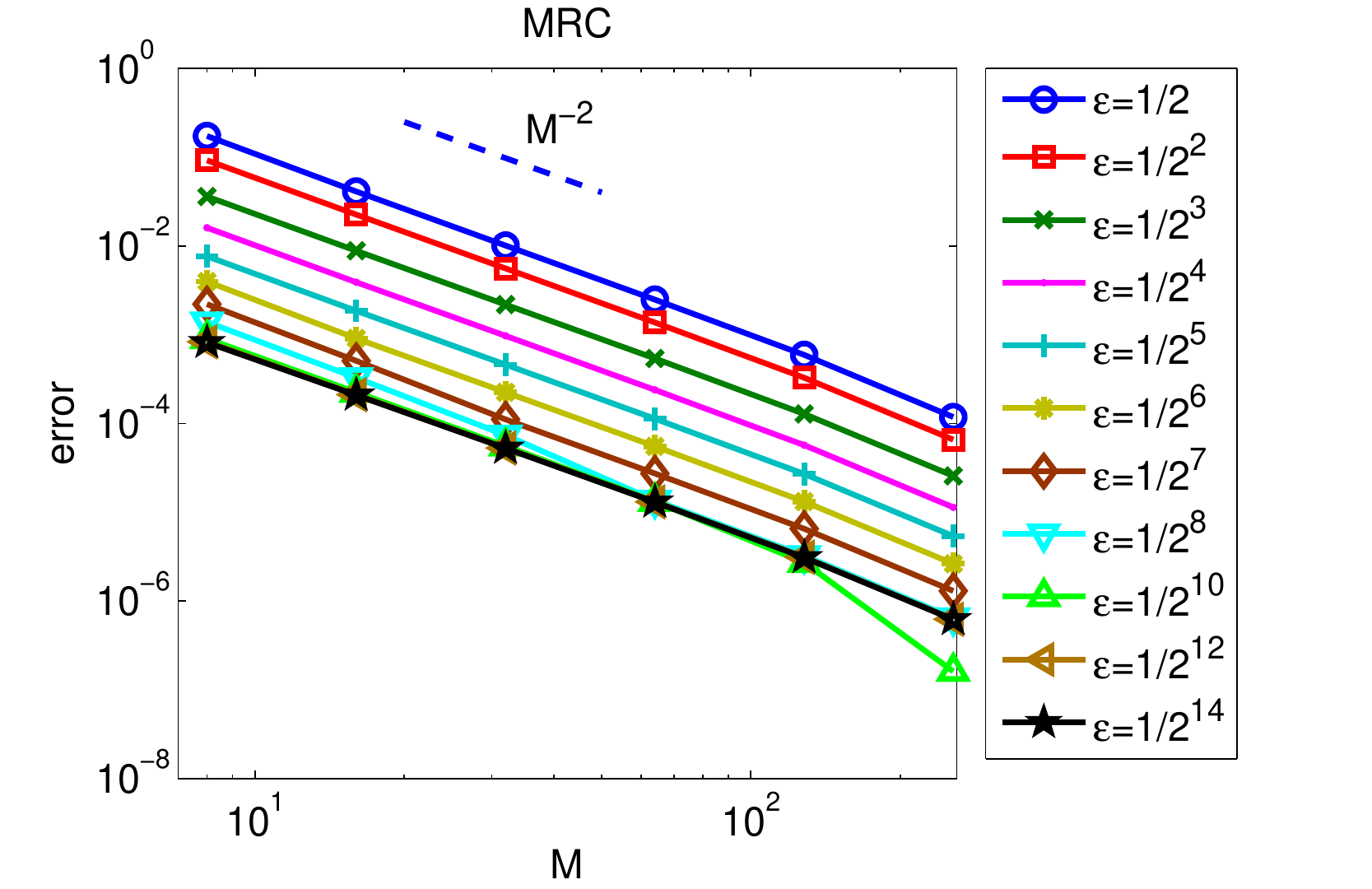,height=5cm,width=7cm}&\psfig{figure=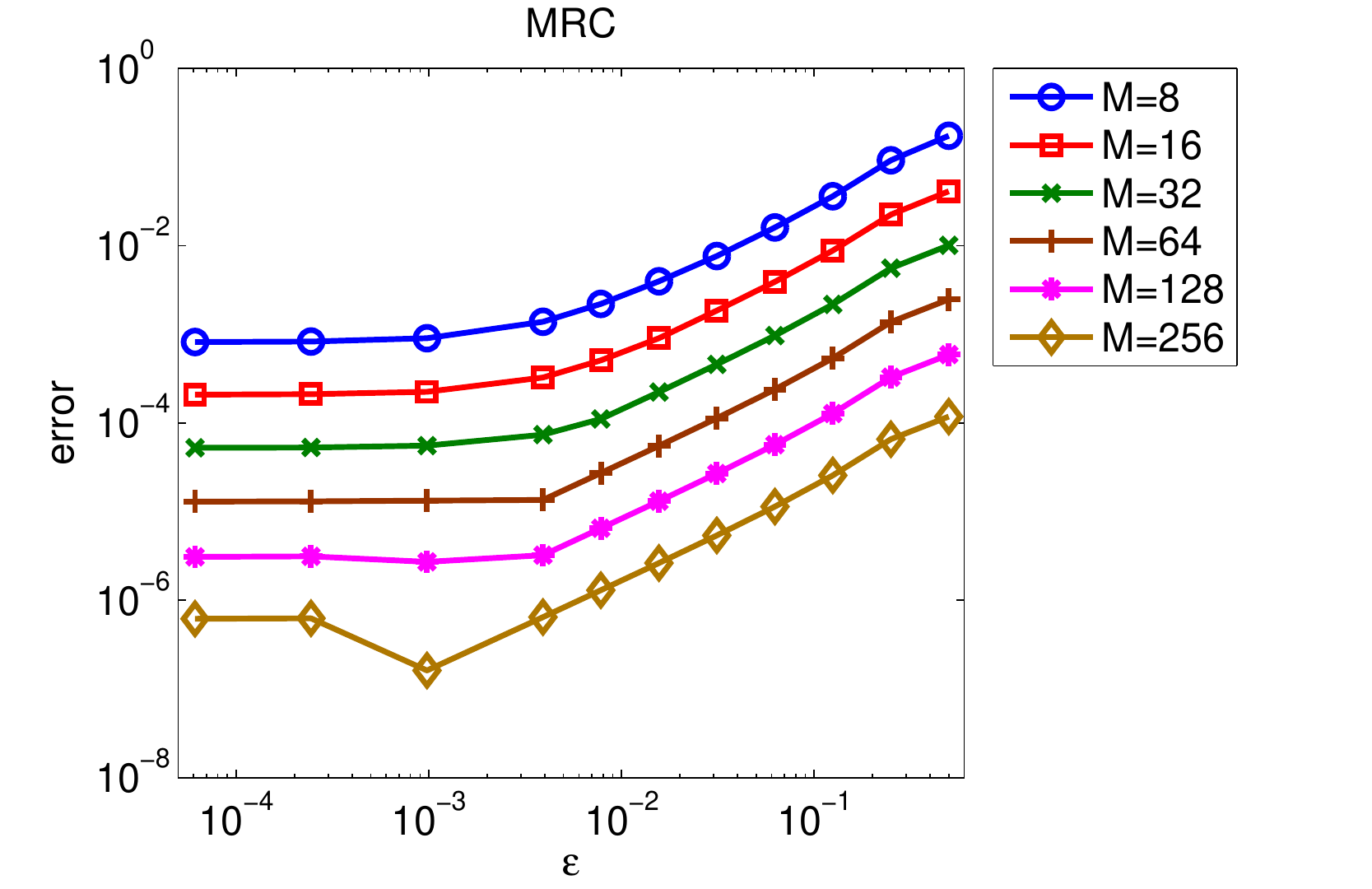,height=5cm,width=7cm}\\
\psfig{figure=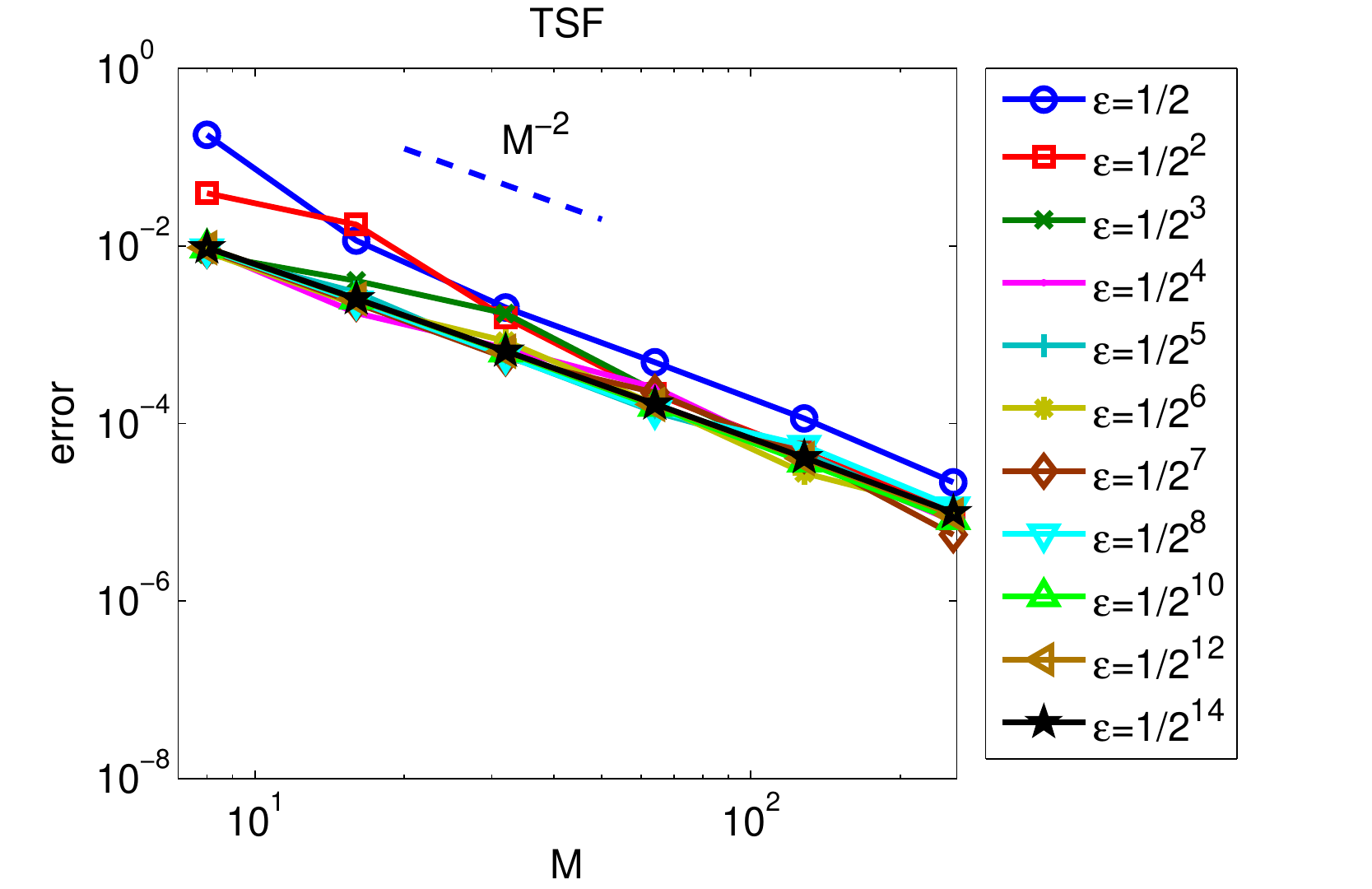,height=5cm,width=7cm}&\psfig{figure=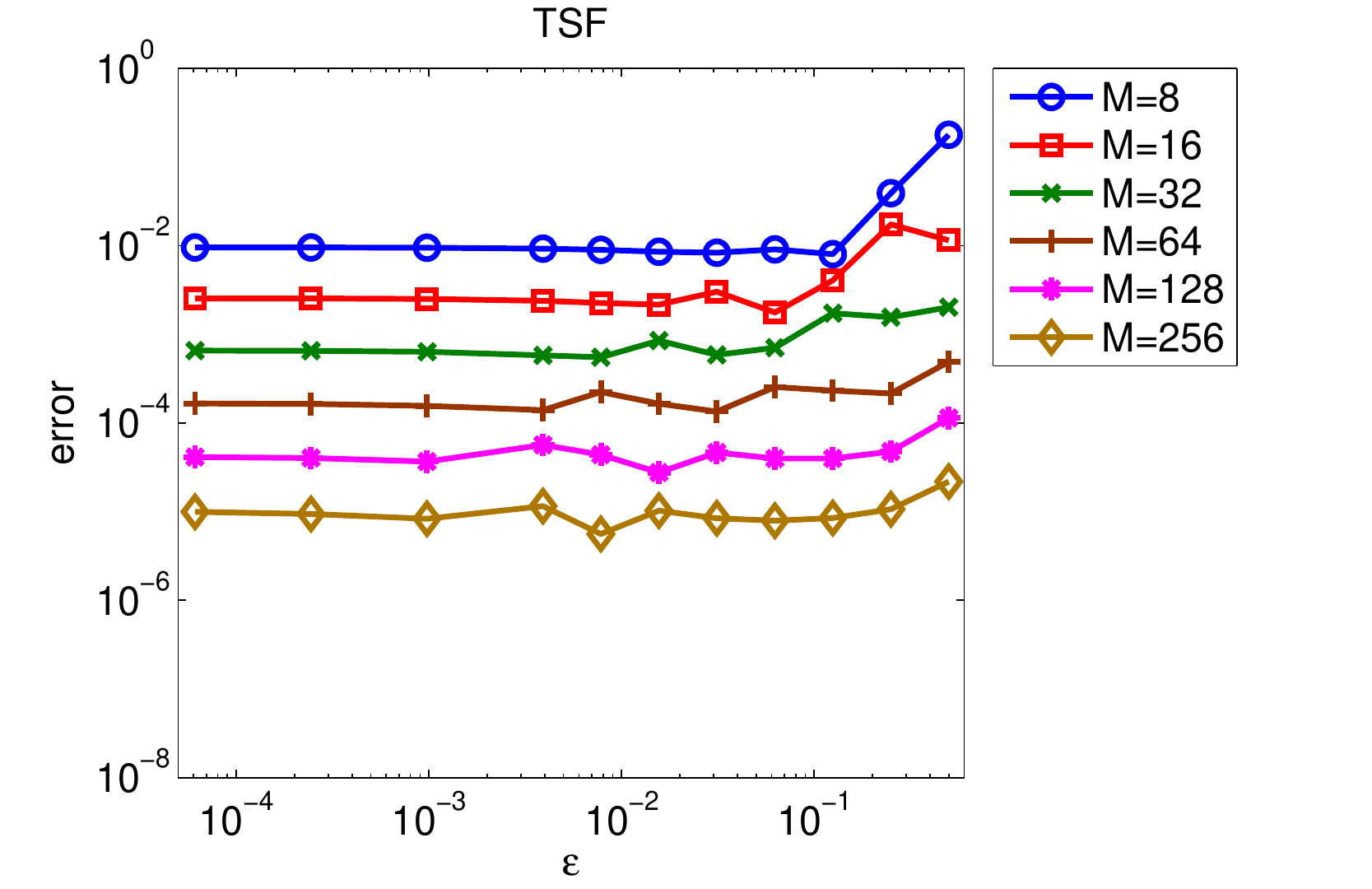,height=5cm,width=7cm}\\
\psfig{figure=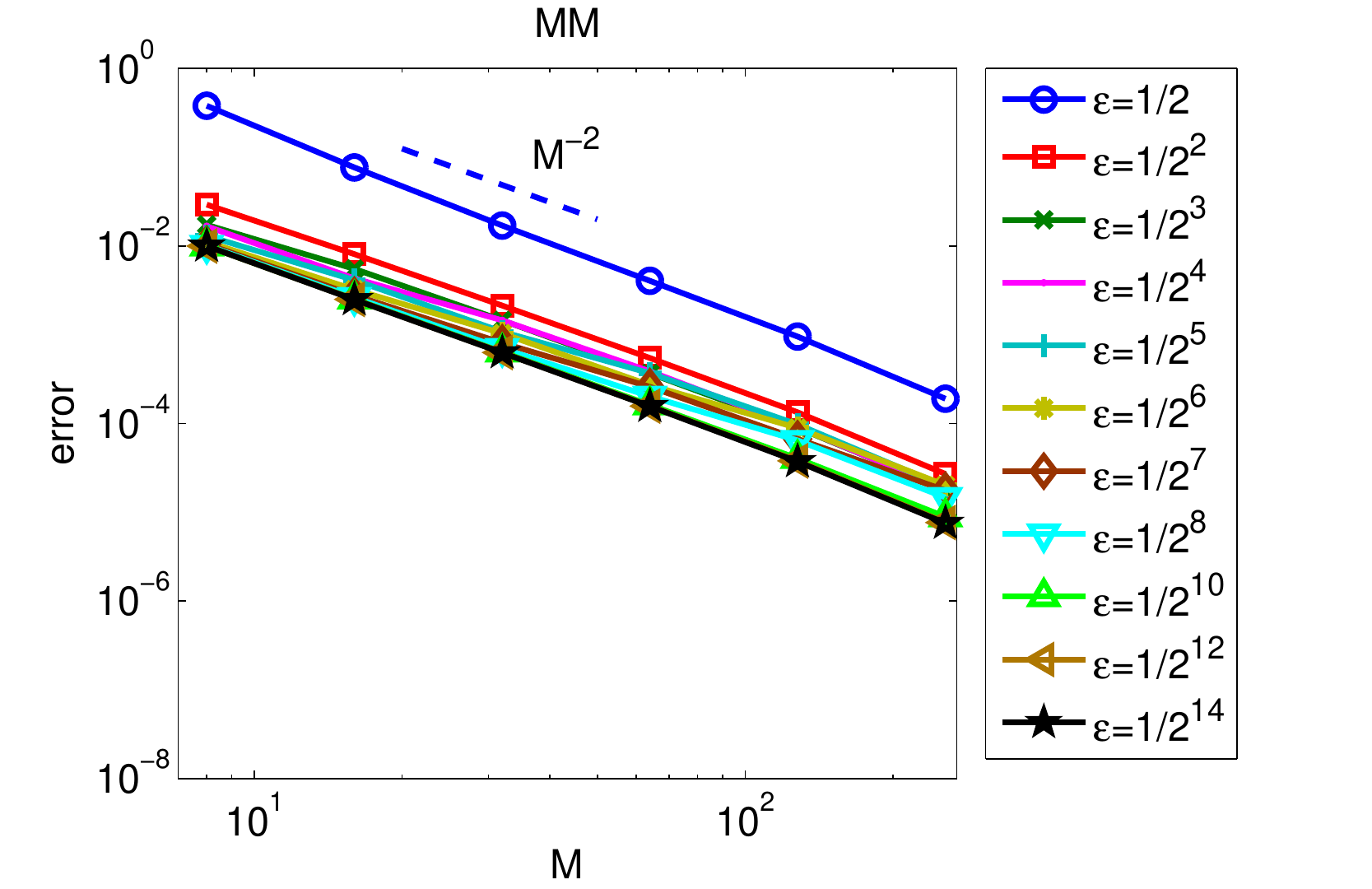,height=5cm,width=7cm}&\psfig{figure=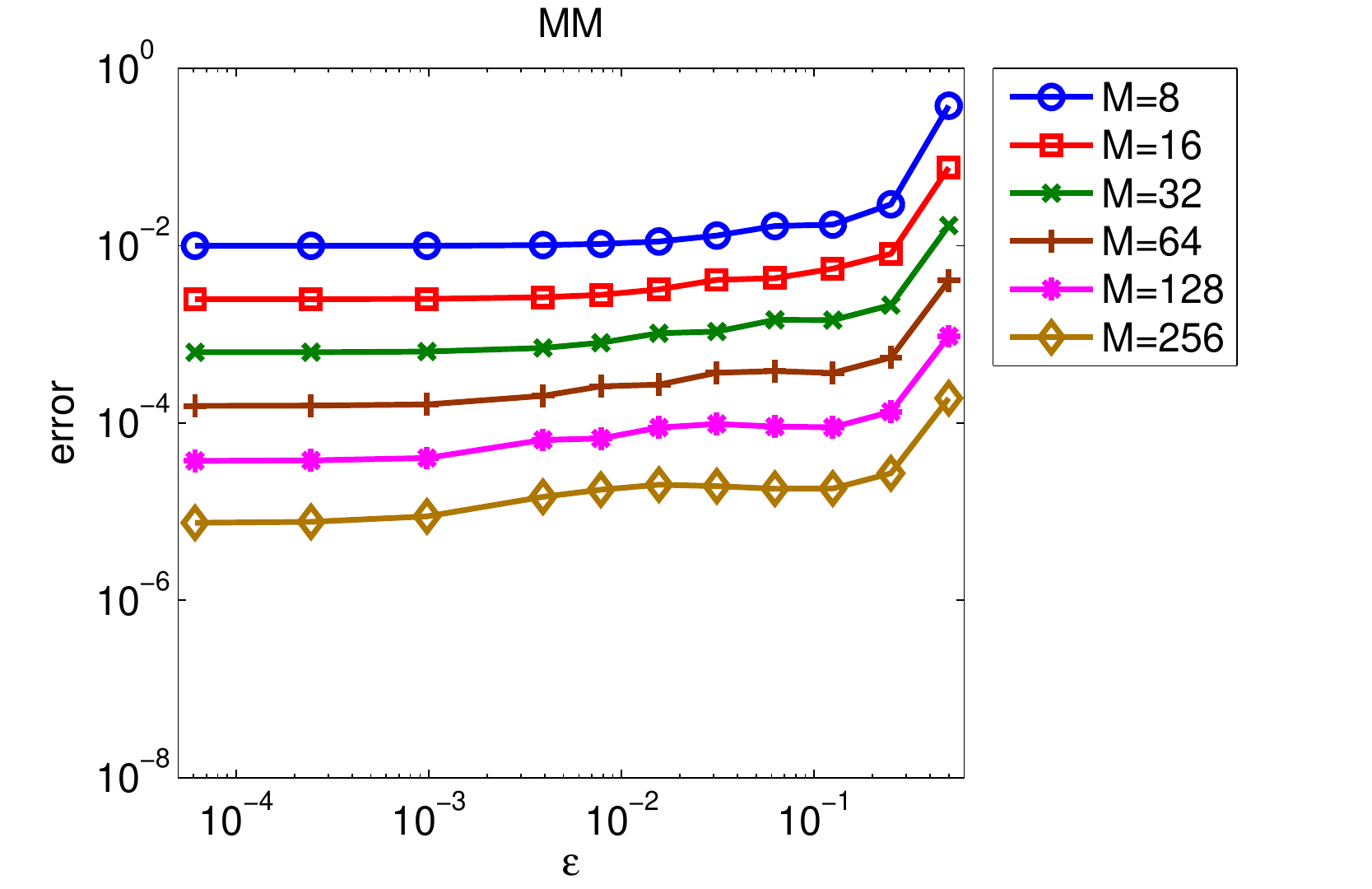,height=5cm,width=7cm}
\end{array}$$
\caption{Errors of MRC, TSF and MM with respect to time steps $M$ under different $\eps$ (left) or with respect to $\eps$ under different $M$ (right) for example \ref{example1}. }\label{fig:example1}
\end{figure}
\begin{figure}[t!]
$$\begin{array}{cc}
\psfig{figure=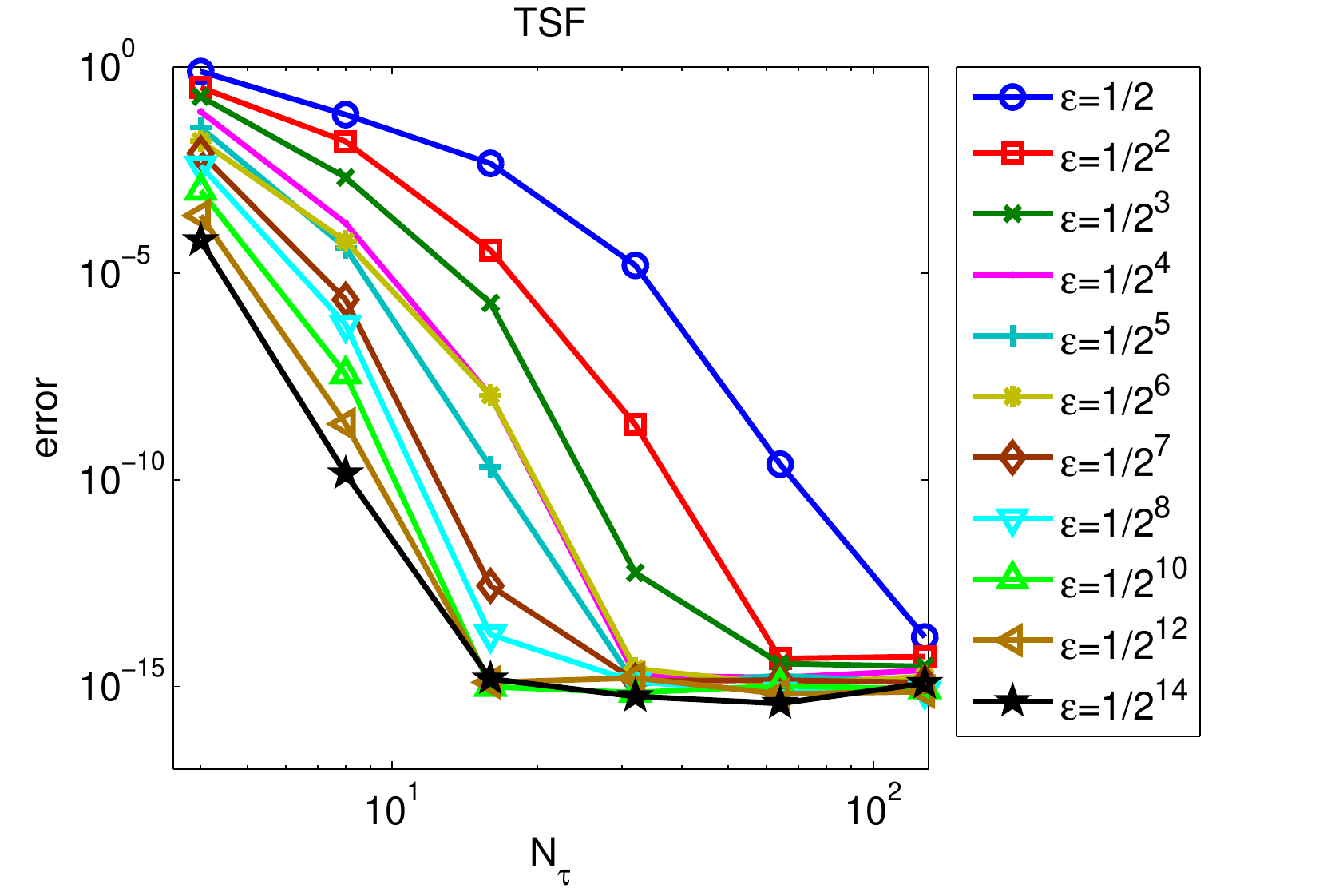,height=5cm,width=7cm}&\psfig{figure=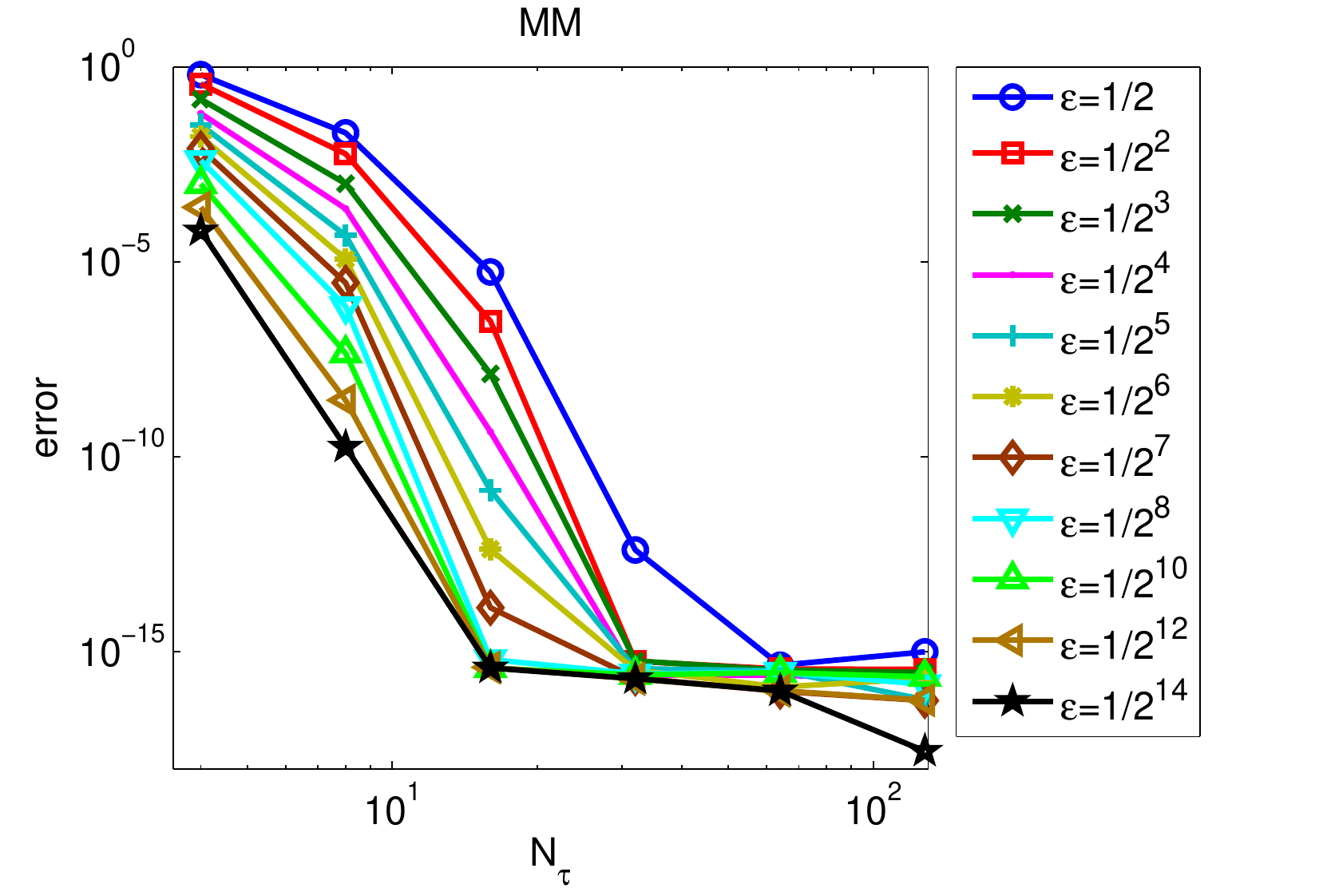,height=5cm,width=7cm}
\end{array}$$
\caption{Error of TSF and MM with respect to $N_\tau$ under different $\eps$ for example \ref{example1}. }\label{fig:tau}
\end{figure}

 \begin{figure}[t!]
$$\begin{array}{cc}
\psfig{figure=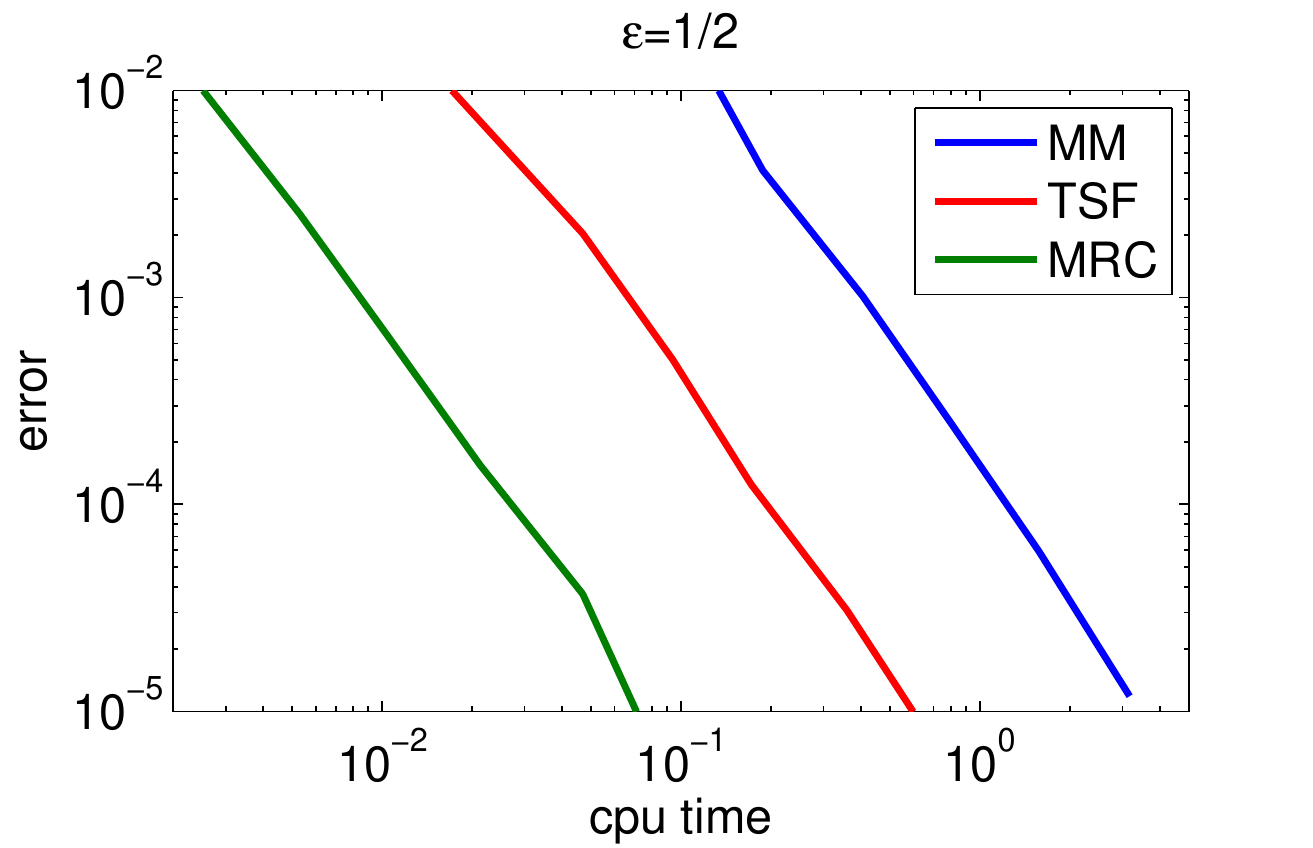,height=5cm,width=7cm}&\psfig{figure=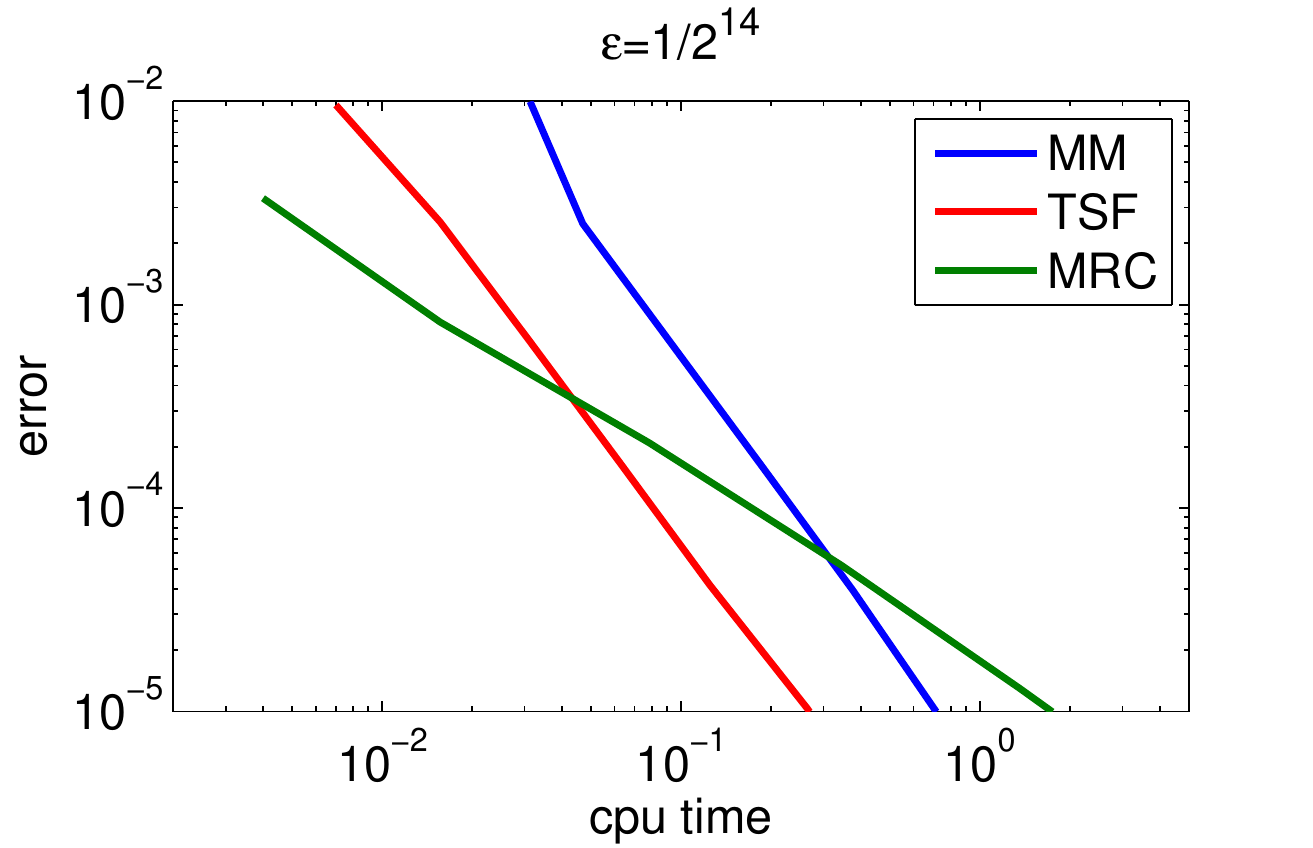,height=5cm,width=7cm}
\end{array}$$
\caption{Efficiency comparison of TSF, MM and MRC in classical (left) or asymptotic regime (right) of $\eps$ for example \ref{example1}: error versus computational time. }\label{fig:com}
\end{figure}

 \begin{figure}[t!]
$$\begin{array}{ccc}
\psfig{figure=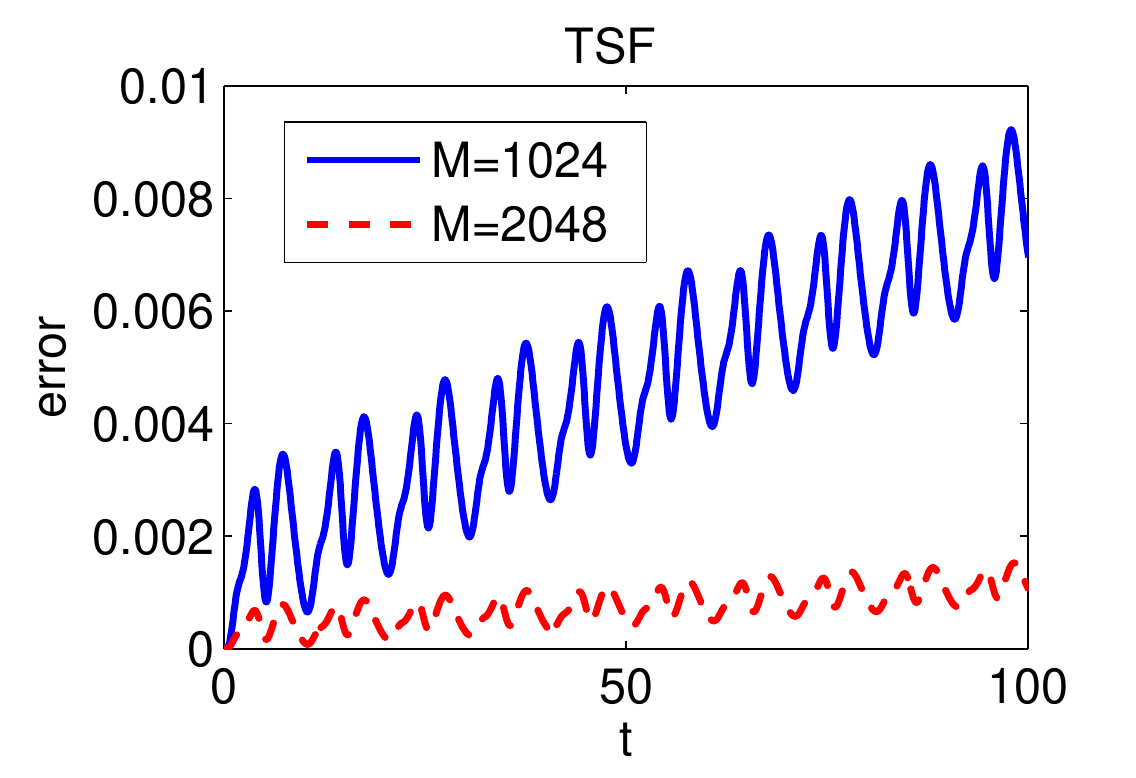,height=5cm,width=4.9cm}&\psfig{figure=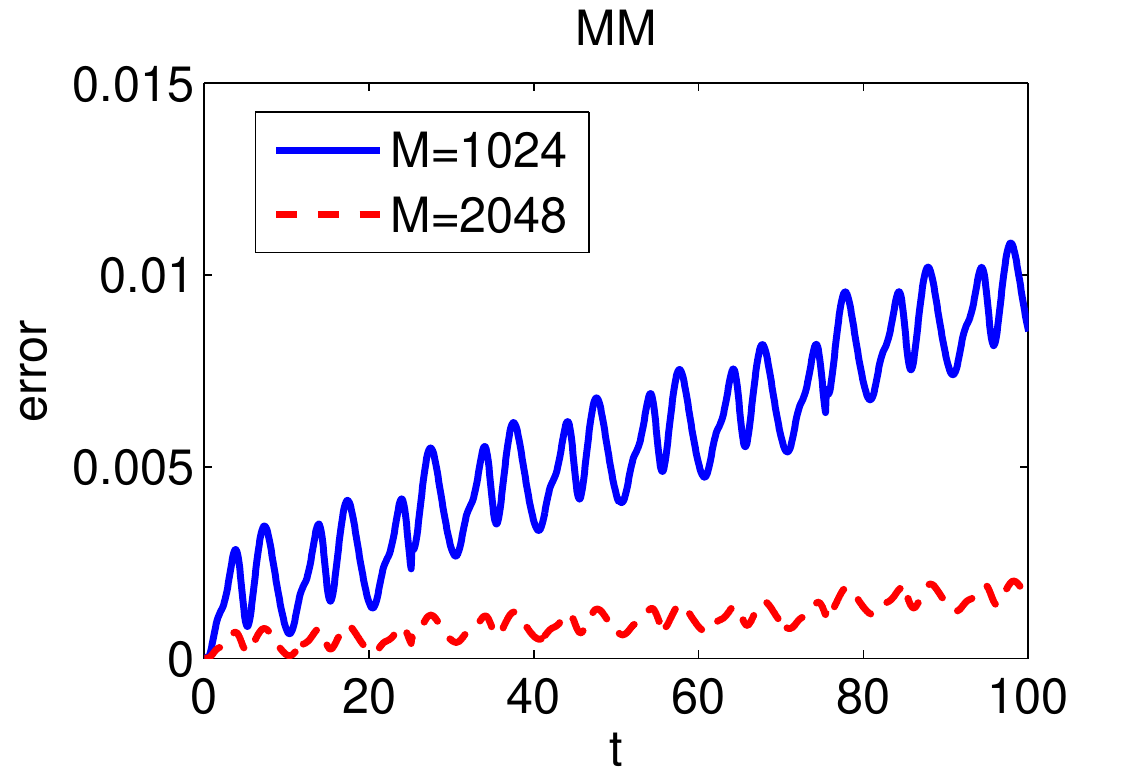,height=5cm,width=4.9cm}
&\psfig{figure=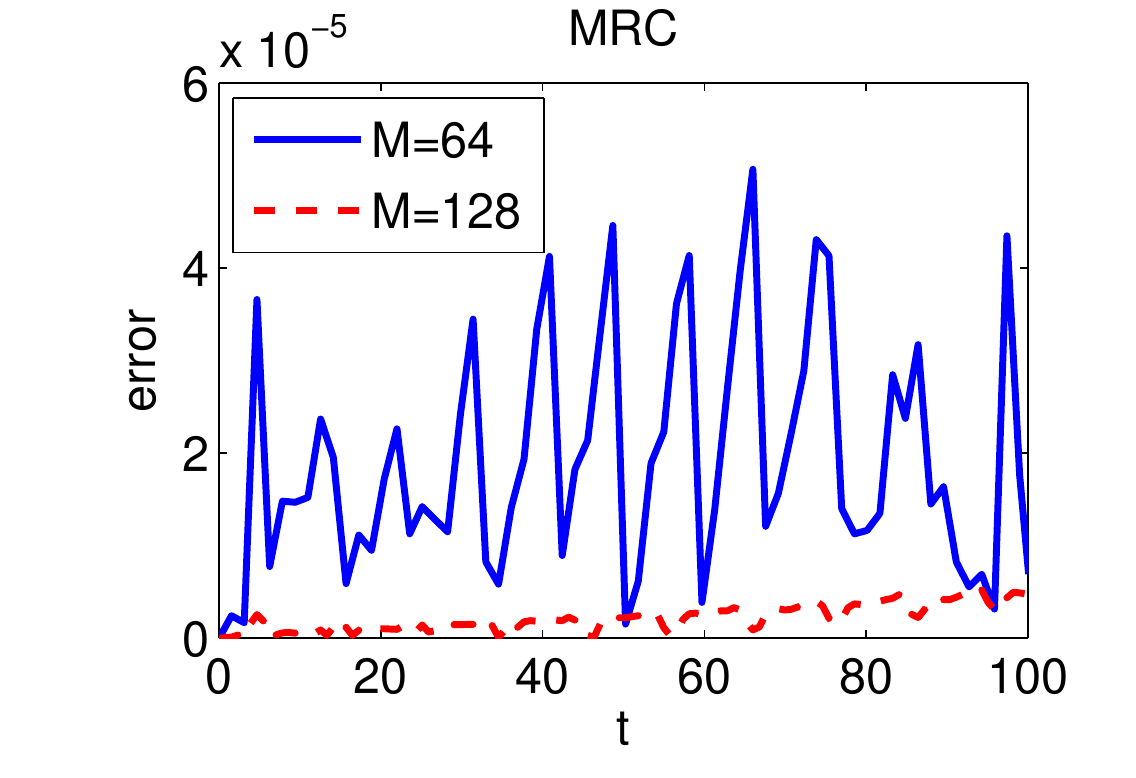,height=5cm,width=4.9cm}
\end{array}$$
\caption{Energy error of TSF, MM (with restart every $T_0=8\pi$) and MRC for example \ref{example1} under $\eps=1/2^{14}$ till $T=32\pi$. $\Delta t=0.0982 $ or $0.0491$ for TSF and MM. }
\label{fig:energy}
\end{figure}


\begin{figure}[t!]
$$\begin{array}{cc}
\psfig{figure=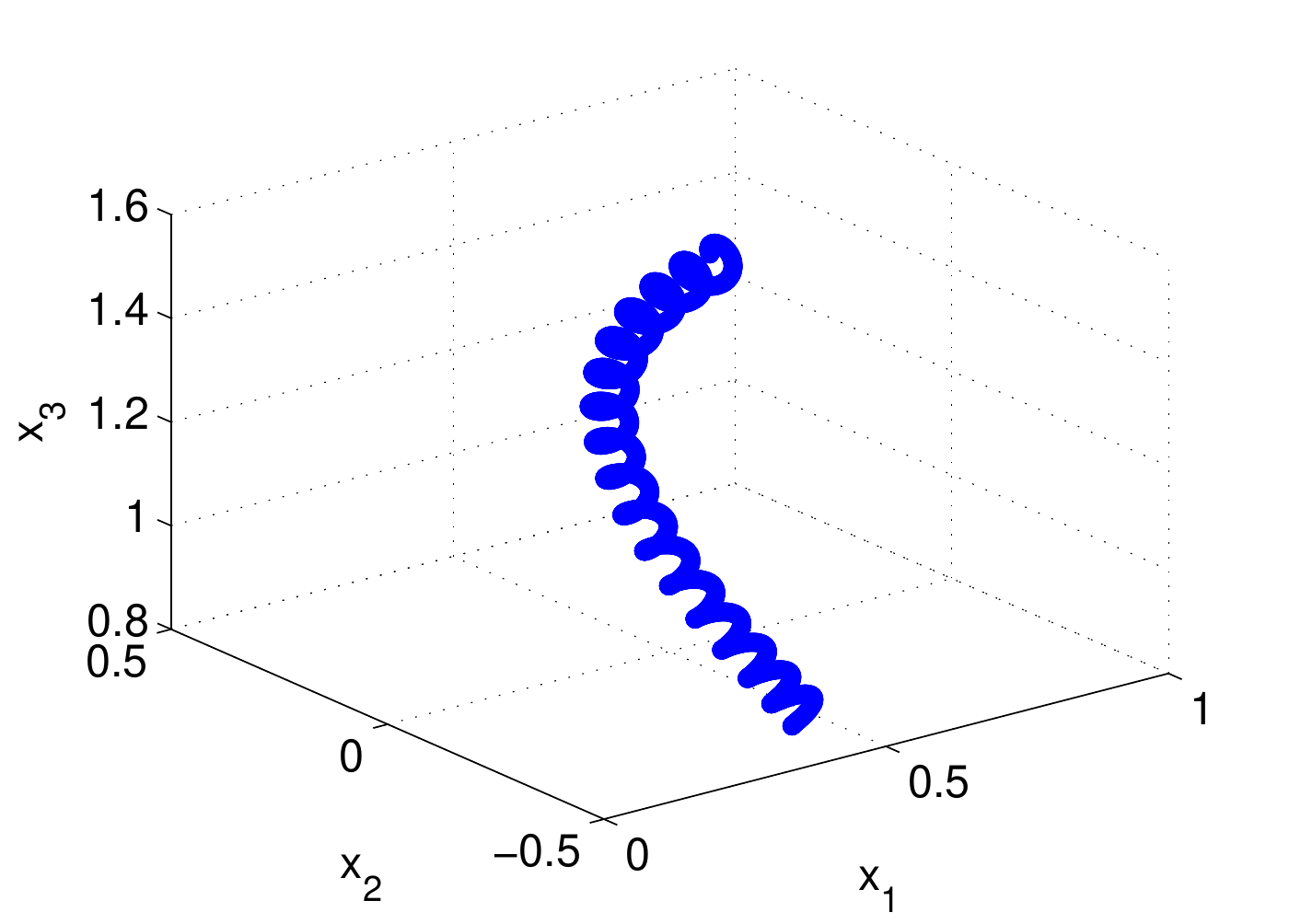,height=5cm,width=7.5cm}
&\psfig{figure=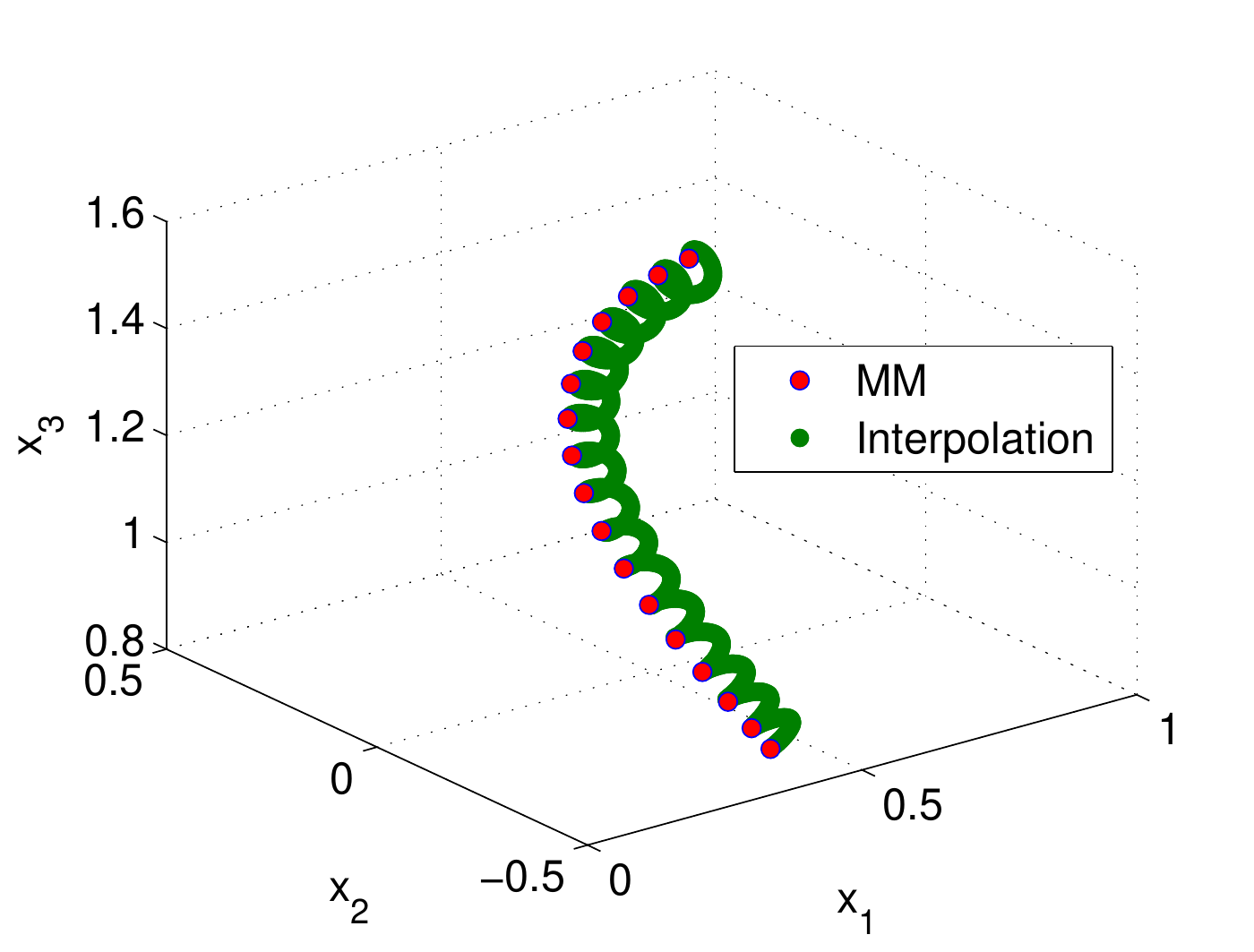,height=5cm,width=7.5cm}
\end{array}$$
\caption{Left: exact trajectory of the particle in example \ref{example1} till $T=\pi$ under $\eps=1/2^5$.
Right: numerical solution of MM under $\Delta t=0.0982$ (red curve) and the fully recovered trajectory with fine linear interpolation.}\label{fig:recovery}
\end{figure}






Finally, we consider the scheme MM ($M=32,\,N_\tau=32$) to illustrate the reconstruction of the whole trajectory for
$t\in [0, \pi]$ (so that $\Delta t=0.0982$).
To do so, we still consider the system \eqref{particle   sysm} with example \ref{example1}, with $\eps=1/2^5$.
In Figure \ref{fig:recovery}, we plot a reference trajectory (obtained with a very small time step) and
the numerical solution obtained by MM using the strategy proposed in Section \ref{sec:mm} (i.e. with a coarse time grid
and using  the linear interpolation strategy in \eqref{full recover}). Using a few grid points,
we can see that the MM method is able to fully restore the complex trajectory (highly oscillatory confined behavior around magnetic field line)
of the particle under trivial computational cost.





\end{example}

\begin{example}\label{example2}\emph{(Varying intensity)}
Secondly, we investigate the numerical performance of the strategy proposed in Section \ref{sec:ex}
for a magnetic field with varying direction and varying intensity on the particle system (\ref{particle   sysm}).

We shall consider the particle system (\ref{example1}) with the same electric field $\bE(\bx)$ as before,
but here the magnetic field is
$$
\bB(\bx)=\left(\begin{split}
&1-\sin(x_2)/2\\
&1+\cos(x_3)/2\\
&1+\cos(x_1)/2
\end{split}\right),\quad \bx=(x_1,x_2,x_3),
$$
which satisfies $\nabla_\bx\cdot\bB=0$ but has a varying intensity in $\bx\in\bR^3$ since
$$
|\bB(\bx)|^2=3+\cos(x_1)+\cos(x_3)-\sin(x_2)+\cos(x_1)^2/4
+\sin(x_2)^2/4+\cos(x_3)^2/4.
$$
We choose the same initial data as before for example \ref{example1} and solve
the problem via the new time formulation (\ref{s}) with the MM method (see Section \ref{sec:ex}).
The reference solution is again obtained by directly solving (\ref{particle   sysm}) with the fourth order
Runge-Kutta method under small step size ($\Delta t=10^{-5}$).

First, we are interested in the error (defined by \eqref{def_err}) against the number of grid points $M$
for the quantities $\bx(t),\,\bv_\parallel(t):=\bv(t)\cdot \bB(\bx(t))\bB(\bx(t)/\|\bB(\bx(t)\|^2$ and $|\bv(t)|$ at $T=1$.
In Figure \ref{fig:xv}, we can observe that the proposed MM scheme converges as number grid points $M$ increases
($\Delta s$ decreases) with uniform second order accurate rate for all $\eps]0,1]$.

Then,  in Figure \ref{fig:energynon}, the time history of the energy error (defined by \eqref{energy error}) of the method with $\Delta s=1/8$ and $\Delta s=1/16$ till a physical time $T=T(s)=100$ and under three different $\eps$ is shown. Let us remark
that the restart strategy is used at every time step. In Figure \ref{fig:energynon}, we observe that the scheme computes the energy (\ref{eneryg par}) with uniform second order accuracy for $\eps\in]0,1]$. Under a rather large step size ($\Delta s\gg\eps$), the scheme is stable in long time computing, and even if a slight linear drift in the energy error is observed, the
energy error  (about $10^{-3}$) is rather good for all $\eps$ considered.
In Figure \ref{fig:energynon}, the relation between the new time $s$ and the physical time $t(s)$ is also plotted
to illustrate that the physical time $t(s)$ is a monotone increasing function.

Finally, we study the dynamics of the magnetic moment defined by
\begin{equation}
\label{magnetic moment}
  I(t)=\frac{1}{2}\frac{|\bv_{\bot}(t)|^2}{|\bB(\bx(t))|},
\end{equation}
which is an analogy of (\ref{magnetic moment PDE}) at the particle level (\ref{particle   sysm}).
We use MM with $\Delta s=1/16$ (so that it is accurate enough) to solve (\ref{particle   sysm}) till $t=100$
with three different values of $\eps$ ($\eps=2^{-9},2^{-10},2^{-11}$).
In Figure \ref{fig:mm}, the relative error on the magnetic moment, i.e. $|I(t_n)-I(0)|/(\eps I(0))$ is displayed as
a function of the rescaled time $s$. Let us remark that the MM scheme captures
this quantity $I(t)$ with uniform second order accuracy for $\eps\in]0,1]$, since $I(t)$ only depends
on $|\bv|$ and $\bv_\parallel$ through $|\bv_{\bot}|^2=|\bv|^2-|\bv_\parallel|^2$.
The deviation of the magnetic moment (\ref{magnetic moment}) behaves as
$|I(t)-I(0)|=O(\eps)$ in the simulation, which is consistent with the results obtained in \cite{Lubich}.
Our scheme captures this adiabatic quantity even when $\Delta s\gg\eps$ whereas the scheme
used in \cite{Lubich} needs $\Delta s<\eps$.





 \begin{figure}[t!]
$$\begin{array}{cc}
&\psfig{figure=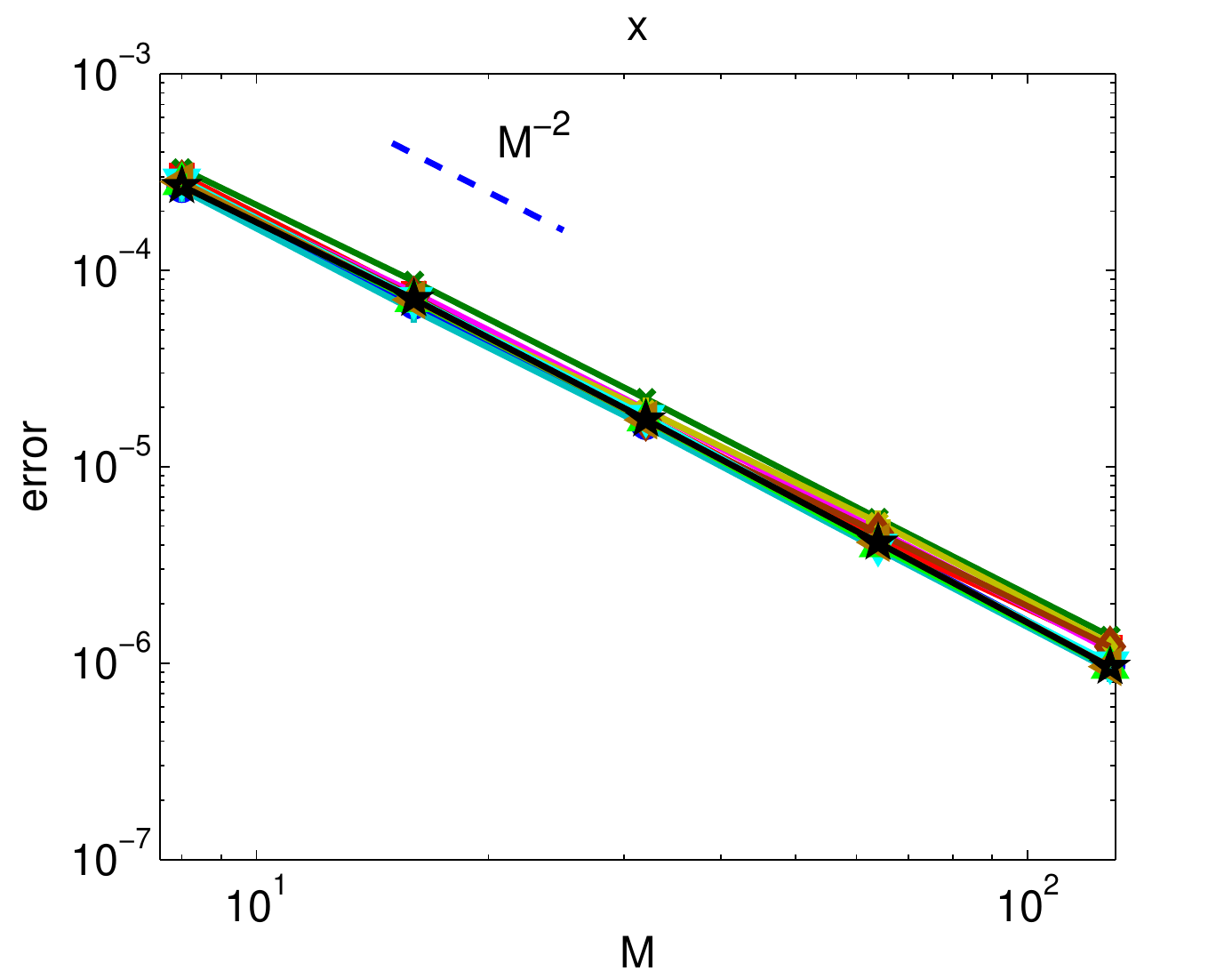,height=5cm,width=4.3cm}
\psfig{figure=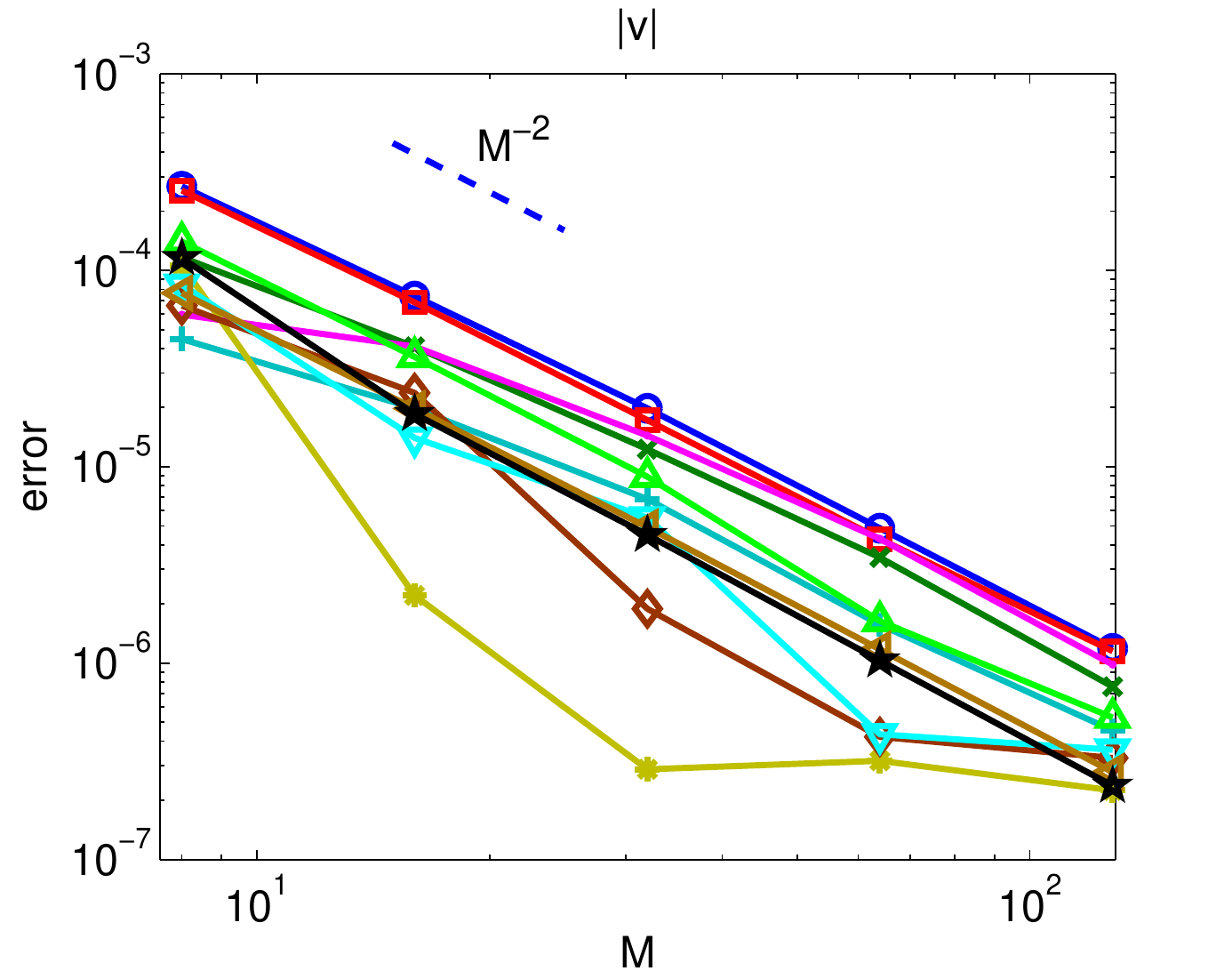,height=5cm,width=4.3cm}
\psfig{figure=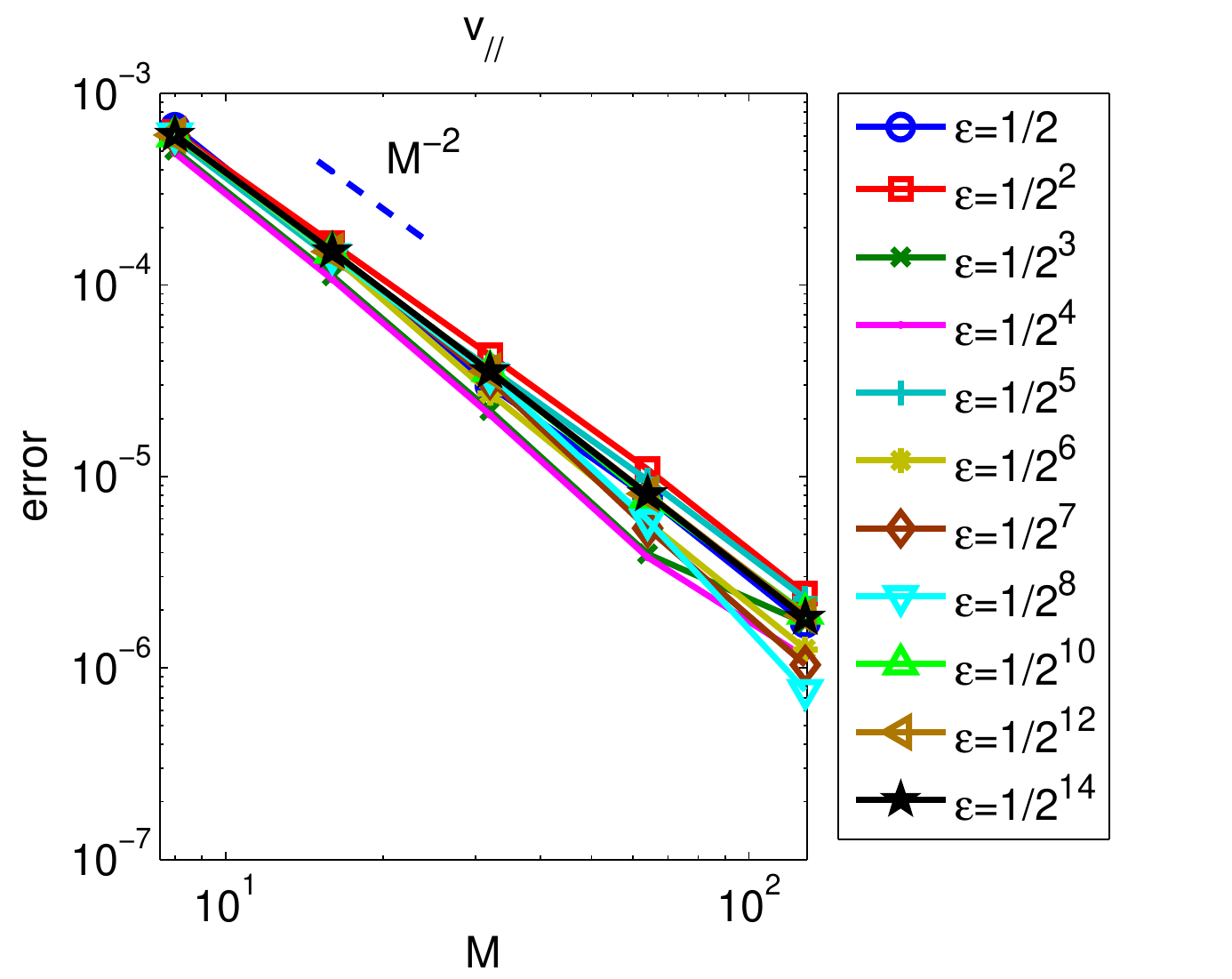,height=5cm,width=5.2cm}
\end{array}$$
\caption{Error of MM under different $M=\Delta s^{-1}$ in $\bx,\,|\bv|$ and $\bv_\parallel$ at $T=1$ in example \ref{example2} of varying intensity. }\label{fig:xv}
\end{figure}
\begin{figure}[t!]
$$\begin{array}{cc}
\psfig{figure=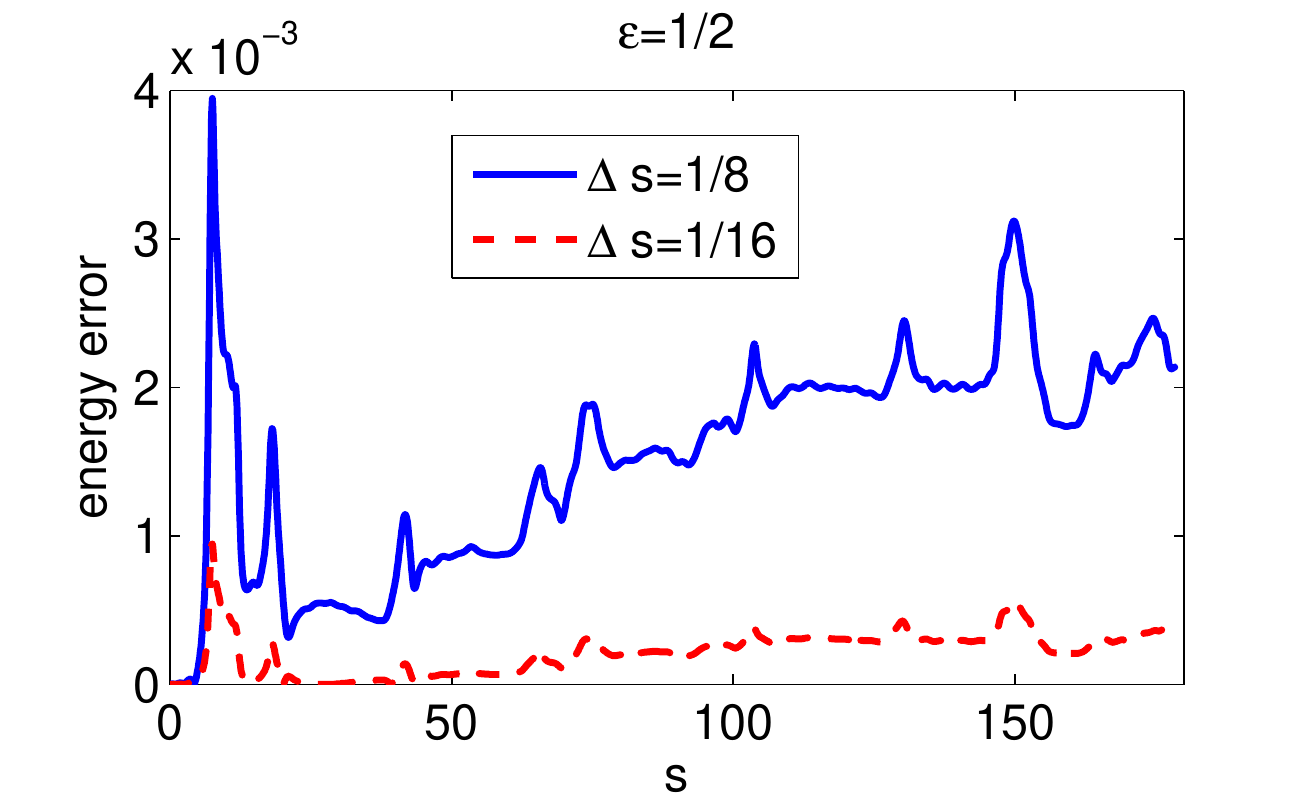,height=4cm,width=7cm}&
\psfig{figure=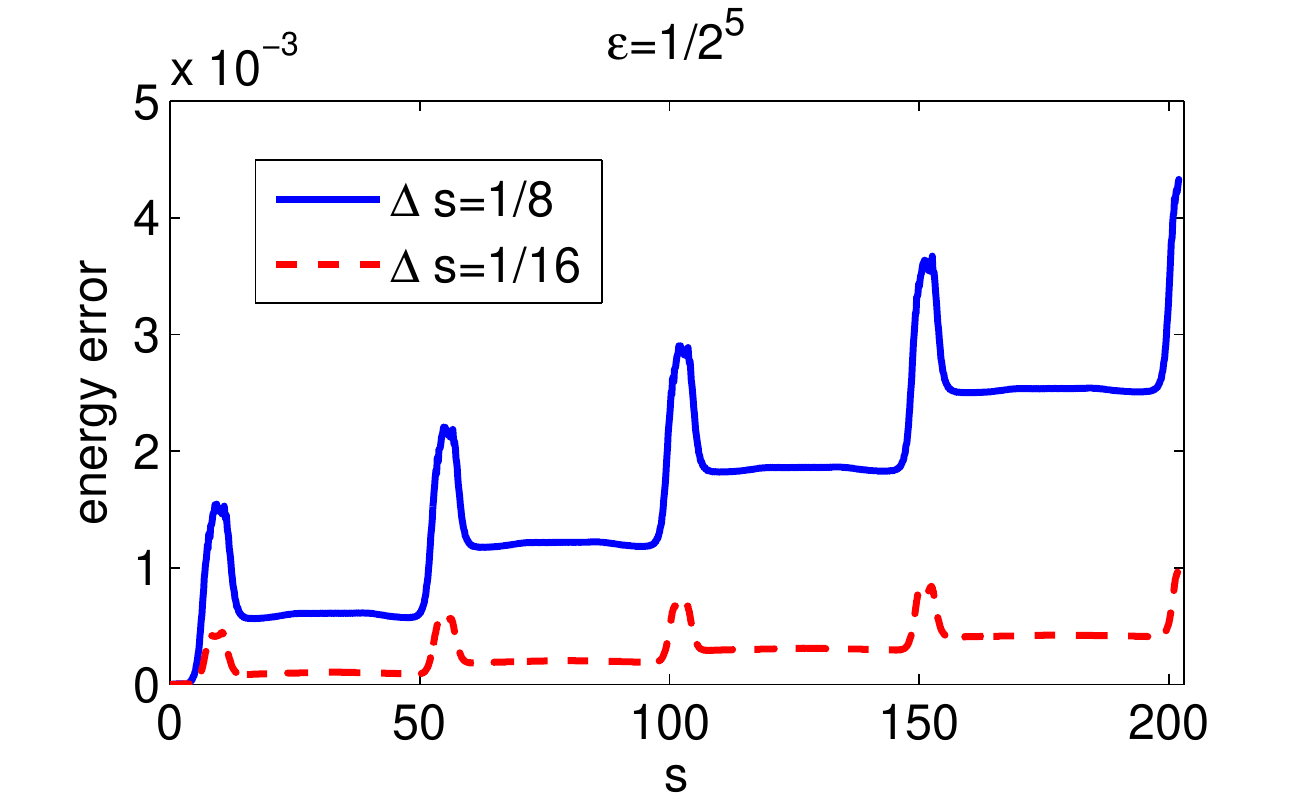,height=4cm,width=7cm}\\
\psfig{figure=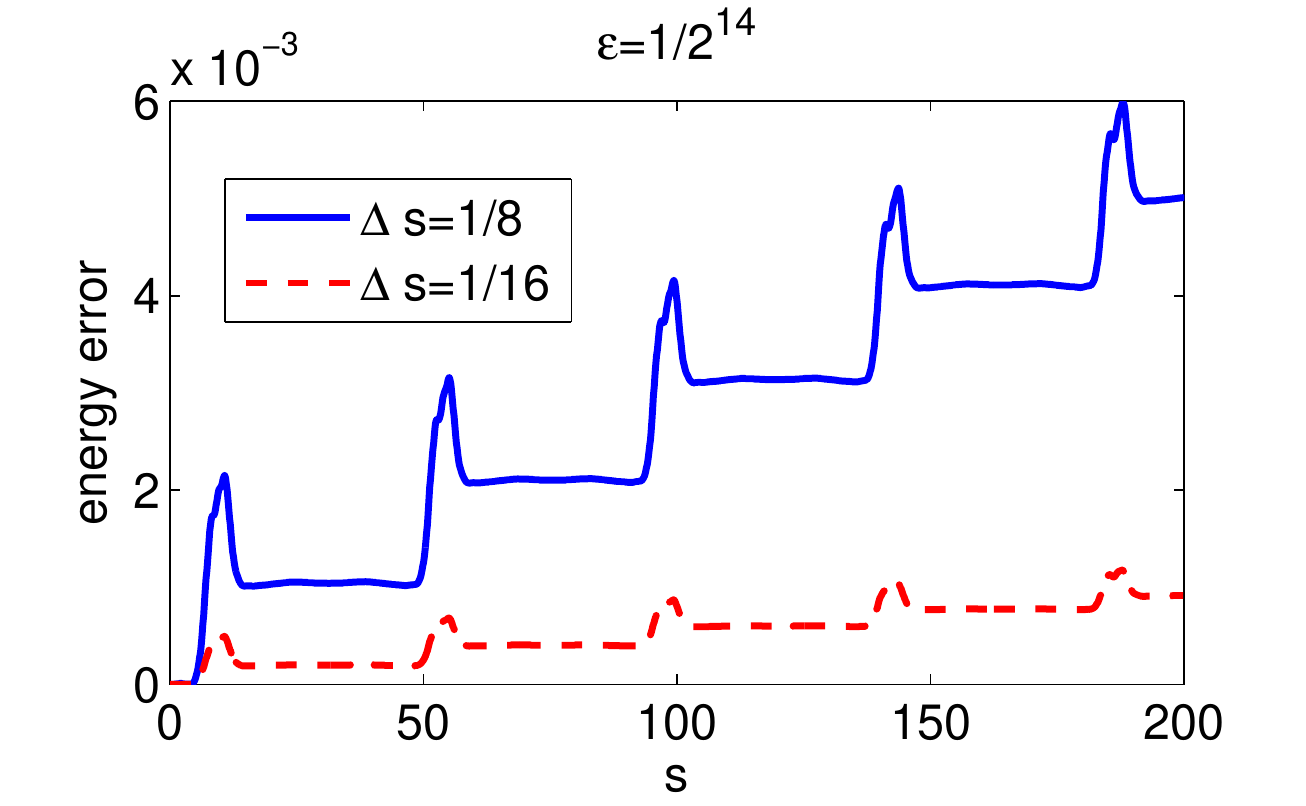,height=4cm,width=7cm}
&\psfig{figure=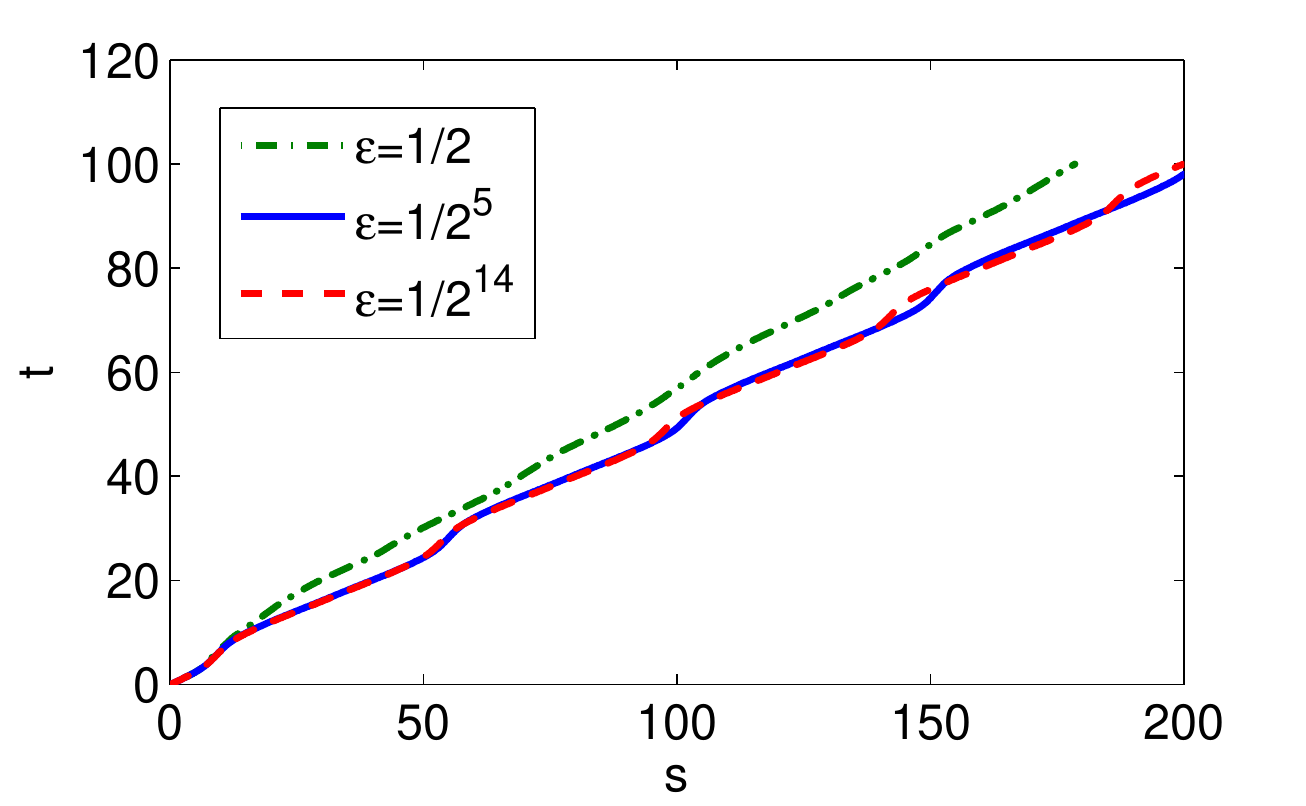,height=4cm,width=7cm}
\end{array}$$
\caption{Energy error of MM (restart each step) for $\eps=1/2,1/2^5,1/2^{14}$ till $t=100$ and the evolution of $t(s)$ in example \ref{example2}. }\label{fig:energynon}
\end{figure}

 \begin{figure}[t!]
$$\begin{array}{cc}
&\psfig{figure=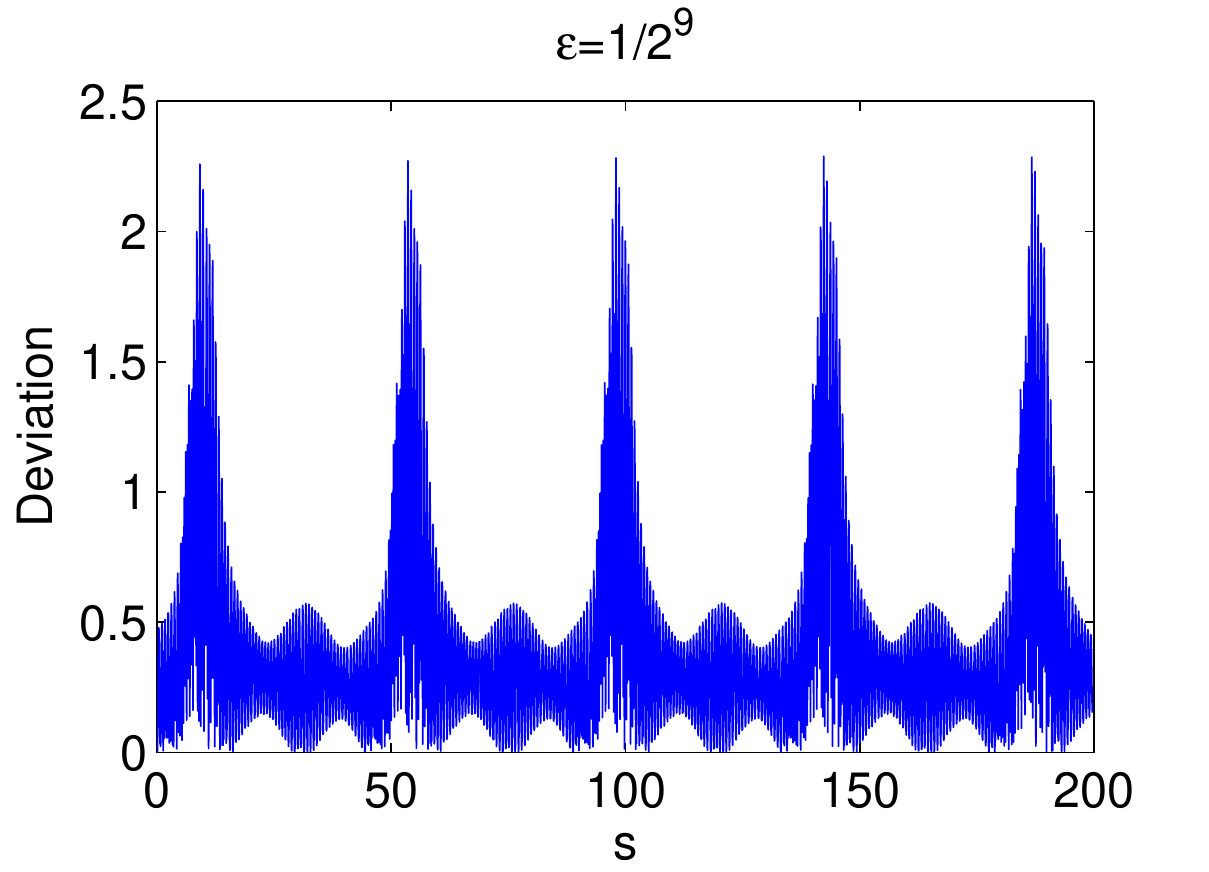,height=5cm,width=4.5cm}
\psfig{figure=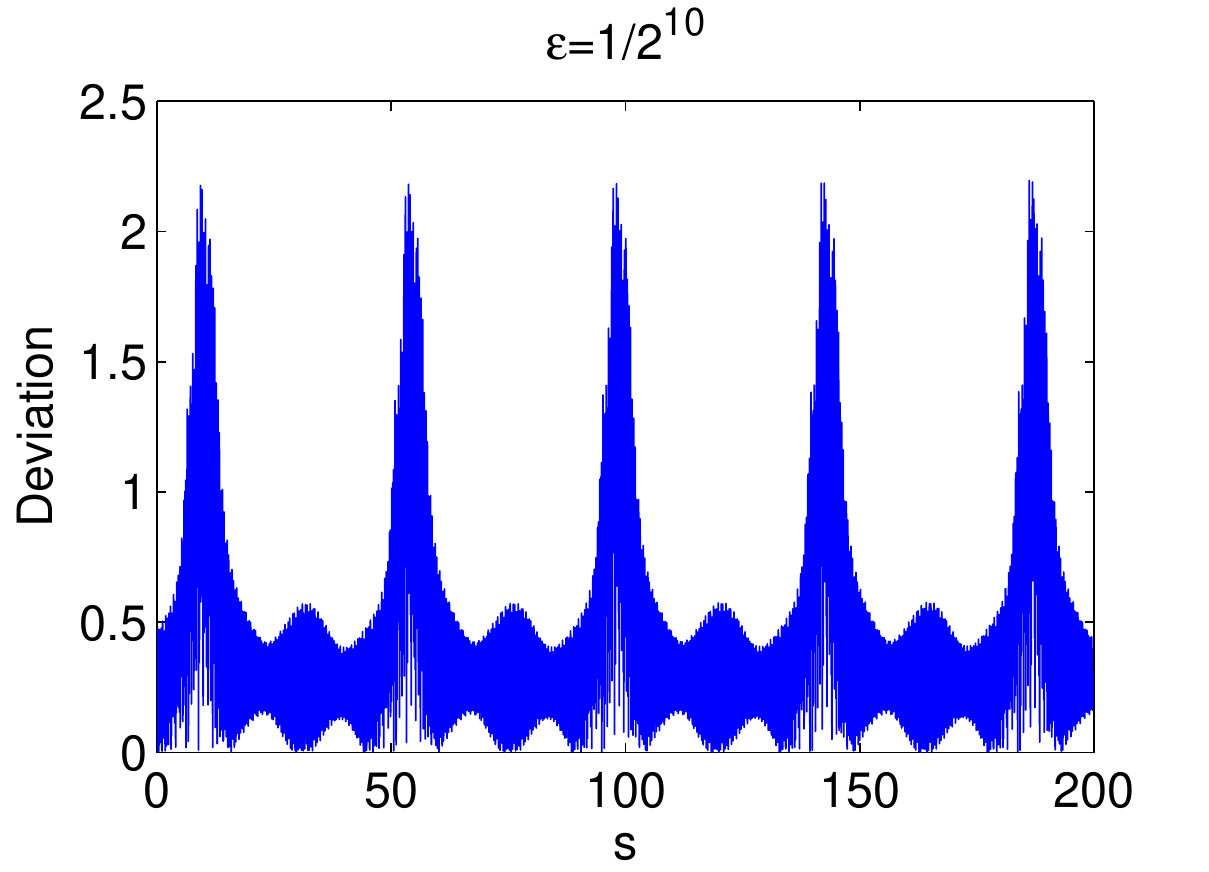,height=5cm,width=4.5cm}
\psfig{figure=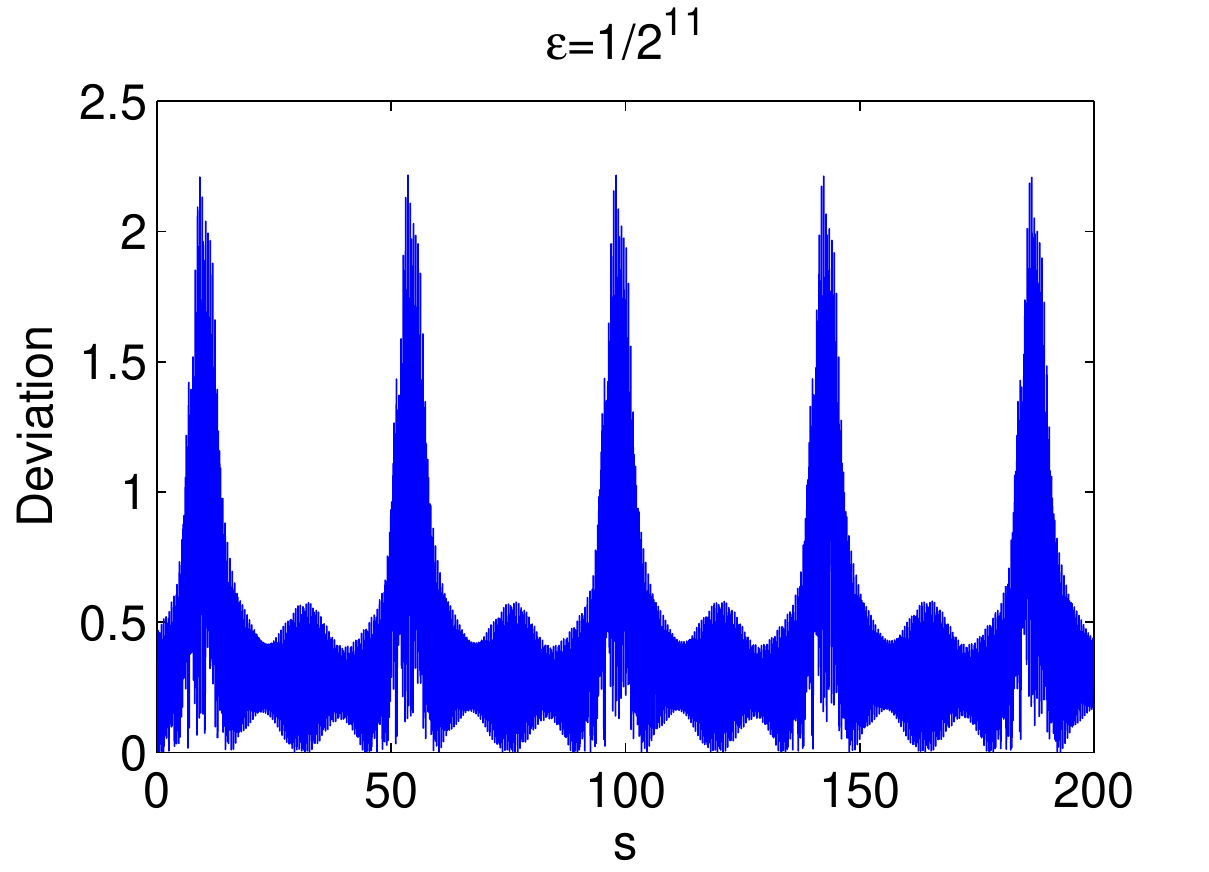,height=5cm,width=4.5cm}
\end{array}$$
\caption{Deviation of the magnetic moment: $\frac{1}{\eps}|I(t)-I(0)|/I(0)$ till $t=100$ in example \ref{example2} under different $\eps$ (computed with $\Delta s=1/16\gg\eps$). }\label{fig:mm}
\end{figure}

\end{example}

\subsection{Simulation of the Vlasov-Poisson system}
In this last part, we focus on the numerical simulation of the full 3D Vlasov-Poisson equation (\ref{eq:1})
using the MRC method. The chosen initial data
is a Maxwellian in velocity and a ring-shape distribution in space with a perturbation in angle \cite{Yang}:
\begin{equation}\label{KH}
f_0(\bx,\bv)=\frac{n_0}{2\pi}\left(1+\eta\cos(k\theta)\right)
\fe^{-5(r-5)^2}\fe^{-\frac{1}{2}|\bv|^2},
  \end{equation}
where $\bx=(x_1, x_2,x_3)$, $\bv=(v_1, v_2,v_3)$,
$r=|\bx|$ and $\theta=\arctan(x_2/x_1)$.
The non-homogenous magnetic field is taken as in \cite{Michel} (screw-pinch setup)
$$
\bB(\bx)=\frac{1}{\sqrt{1+\alpha^2x_1^2+\alpha^2x_2^2}}\left(\begin{array}{c}
 \alpha x_2 \\
-\alpha x_1\\
1
\end{array}\right),$$
which satisfies both $|\bB(\bx)|=1$ and $\nabla_\bx\cdot\bB(\bx)=0$.
The spatial domain is a cartesian geometry $\bx=(x_1, x_2,x_3) \in\Omega=[-8,8]\times[-8,8]\times[0,1]$. We choose
$n_0=100,\eta=0.05, k=4$ and discretize the spatial domain $\Omega$ with $N_{x_1}=N_{x_2}=256$ points in
$x_1,x_2$-directions and $N_{x_3}=4$ points in $x_3$-direction.
As a diagnostic, we consider the following quantity:
\begin{align*}
&\rho^\eps(t,\bx)=\int_{\bR^3}f^\eps(t,\bx,\bv)d\bv,\qquad \bx\in\Omega.
\end{align*}
For the PIC method, we choose $N_p=100 \times N_{x_1}N_{x_2}N_{x_3}$ particles and
the projection of the particles on the spatial grid is done by cubic splines.

In Figures \ref{fig:3d1}, the density $\rho^\eps$ is displayed at different times for $\eps=1/2^5$ with $M=256$
whereas $\alpha=0$ in the magnetic field, so that $\bB$ is homogeneous and aligned with the $x_3$ direction.
There is two different dynamics which can be seen in the results: an instability
develops in the direction orthogonal to the magnetic field (one can see four vortices at time $t=64\pi$)
and a slight parallel dynamics in the plane parallel to the magnetic field. In Figures \ref{fig:3d2} and \ref{fig:3d3},
a non-homogeneous magnetic field is considered
($\alpha=0.003$). We can observe that the dynamics is different dynamics from the homogeneous case.
Indeed, the instability leading to the formation of four vortices is different and one can see stronger non homogeneous
phenomena in the $x_3$ direction due to the expression of the magnetic field.

Finally, in Figure \ref{fig:energy3d}, we plot the energy error for both configurations ($\alpha=0$ and $\alpha=0.003$).
Very good conservations are obtained for long time. Moreover, we consider the relative error between the
Vlasov-Poisson system (\ref{eq:1}) and the asymptotic model (\ref{eq:lm}) as a function of $\eps$, for $\alpha=0.003$.
To do so, we compute the $L^\infty$ norm (in space) of $|\rho^\eps(t=\pi,\bx)-\rho(t=\pi,\bx)|/|\rho^\eps(t=\pi,\bx)|$ at the final time $t=\pi$.
We can see that when $\eps$ decreases, the error is $O(\eps)$, as predicted by the theory.


\begin{figure}[h]
$$\begin{array}{cc}
 \includegraphics[width=7cm,height=4.3cm]{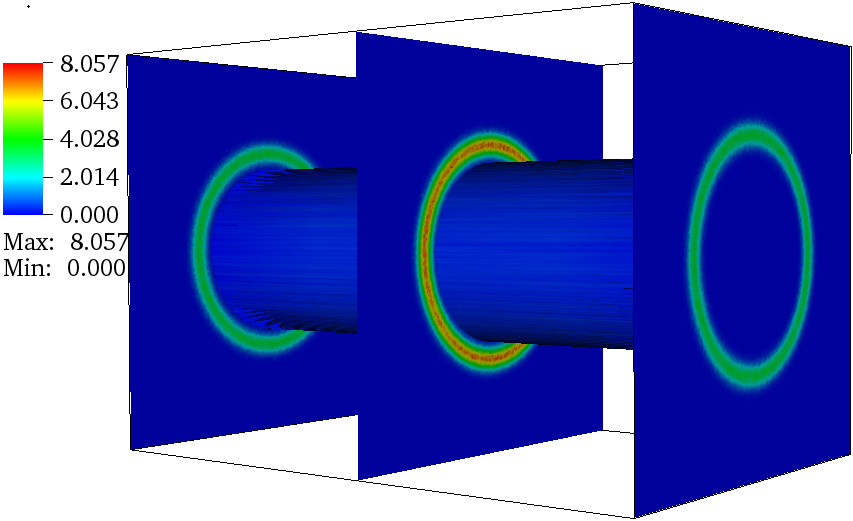}& \includegraphics[width=7cm,height=4.3cm]{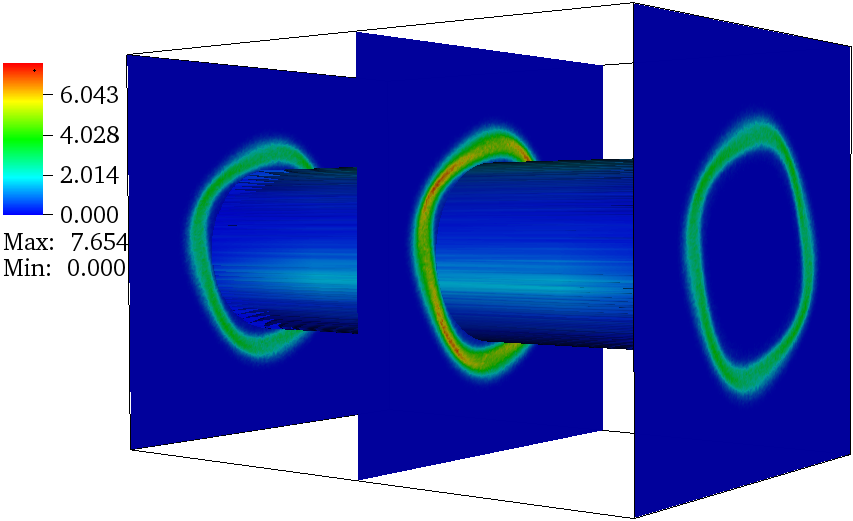}\\
\includegraphics[width=7cm,height=4.3cm]{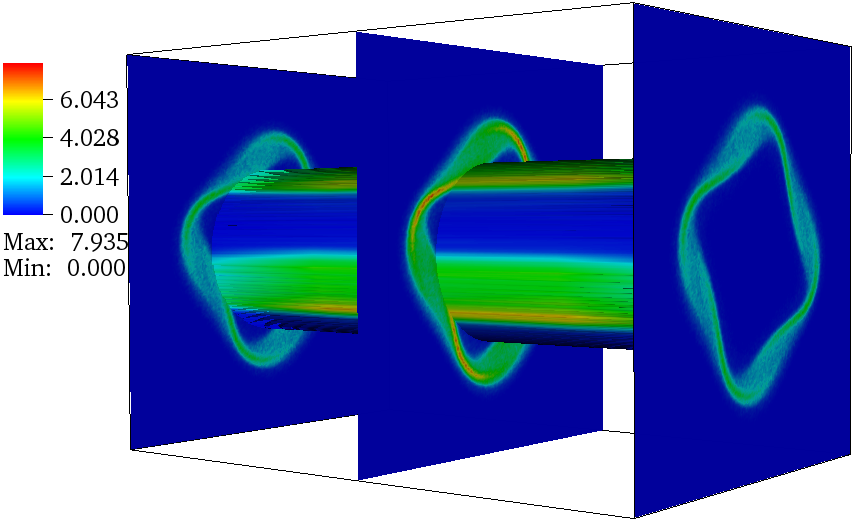}&  \includegraphics[width=7cm,height=4.3cm]{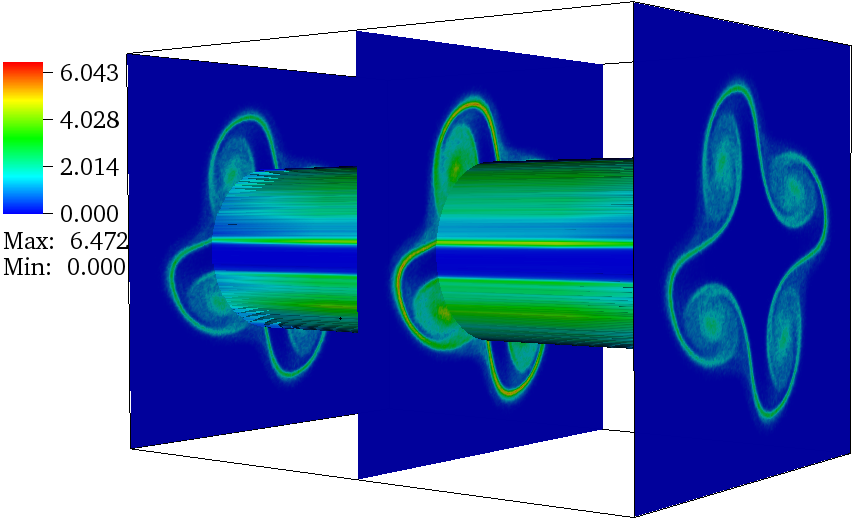}
  \end{array}$$
\caption{Vlasov-Poisson case: pseudo-color snapshots of $\rho^\eps$ under $\eps=1/2^5$ at $t=0,16\pi,32\pi,64\pi$  with initial condition \ref{KH} with
$\alpha=0$. }\label{fig:3d1}
\end{figure}
\begin{figure}[h]
$$\begin{array}{cc}
 \includegraphics[width=7cm,height=4.3cm]{3d1.png}& \includegraphics[width=7cm,height=4.3cm]{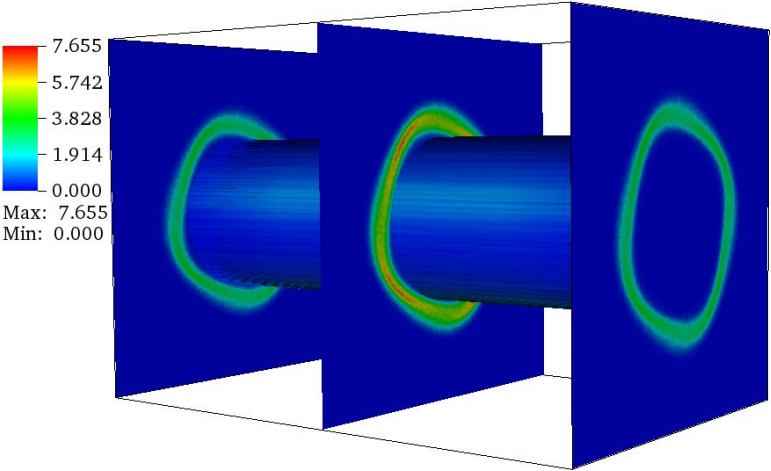}\\
\includegraphics[width=7cm,height=4.3cm]{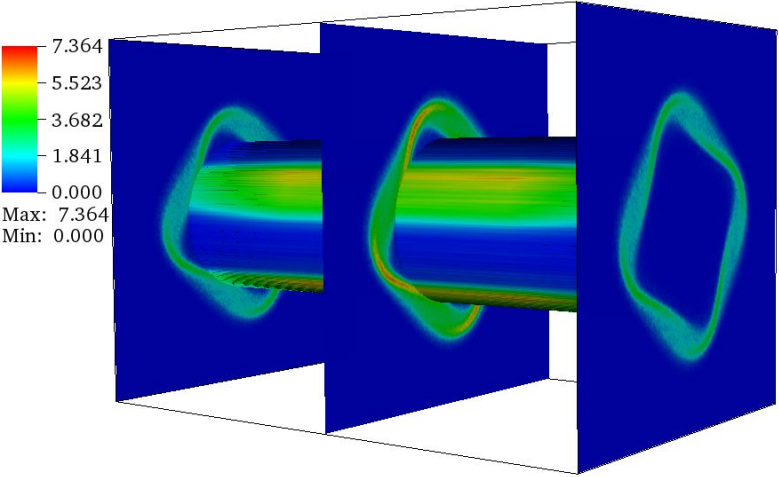}&  \includegraphics[width=7cm,height=4.3cm]{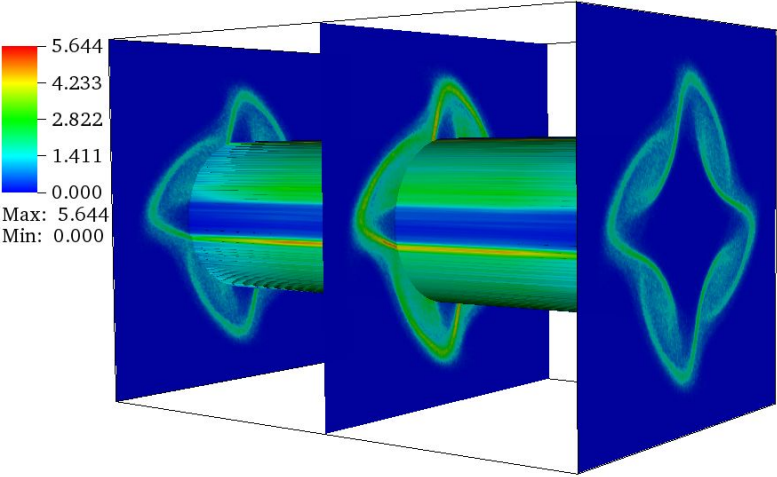}
  \end{array}$$
\caption{Vlasov-Poisson case: pseudo-color snapshots of $\rho^\eps$ under $\eps=1/2^5$ at $t=0,16\pi,32\pi,64\pi$ with initial condition \ref{KH} with
$\alpha=0.003$. }\label{fig:3d2}
\end{figure}

\begin{figure}[h]
$$\begin{array}{cc}
\includegraphics[width=7cm,height=4.3cm]{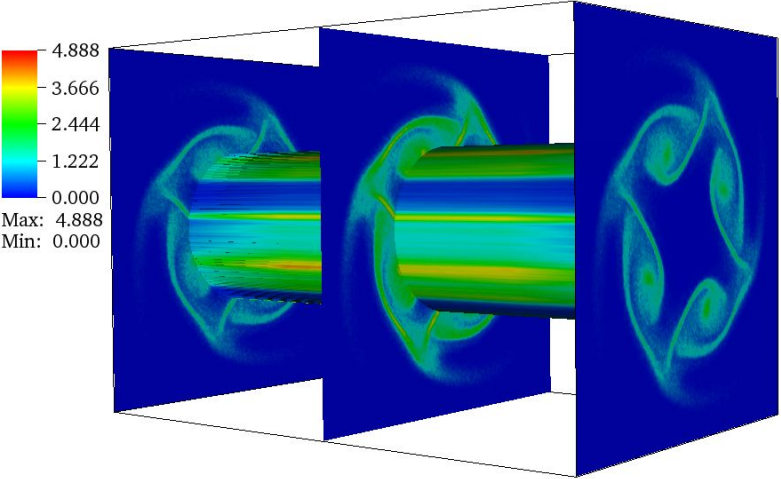}&  \includegraphics[width=7cm,height=4.3cm]{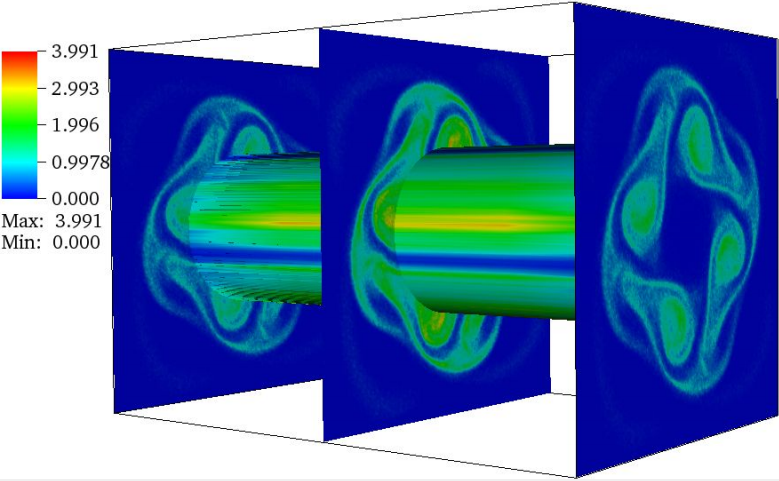}
  \end{array}$$
\caption{Vlasov-Poisson case: pseudo-color snapshots of $\rho^\eps$ under $\eps=1/2^5$ at $t=88\pi,128\pi$ with initial condition \ref{KH} with $\alpha=0.003$. }\label{fig:3d3}
\end{figure}

\begin{figure}[t!]
$$\begin{array}{cc}
\psfig{figure=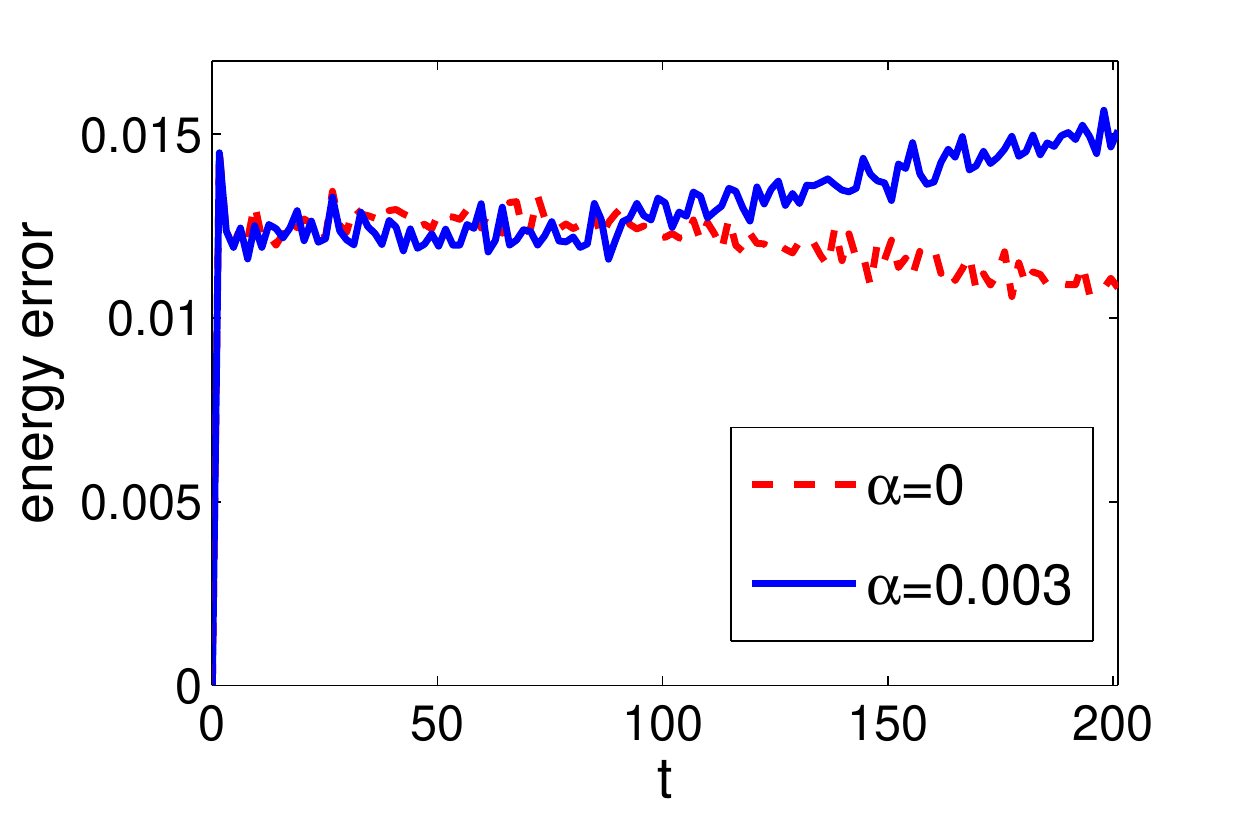,height=4cm,width=7cm}&
\psfig{figure=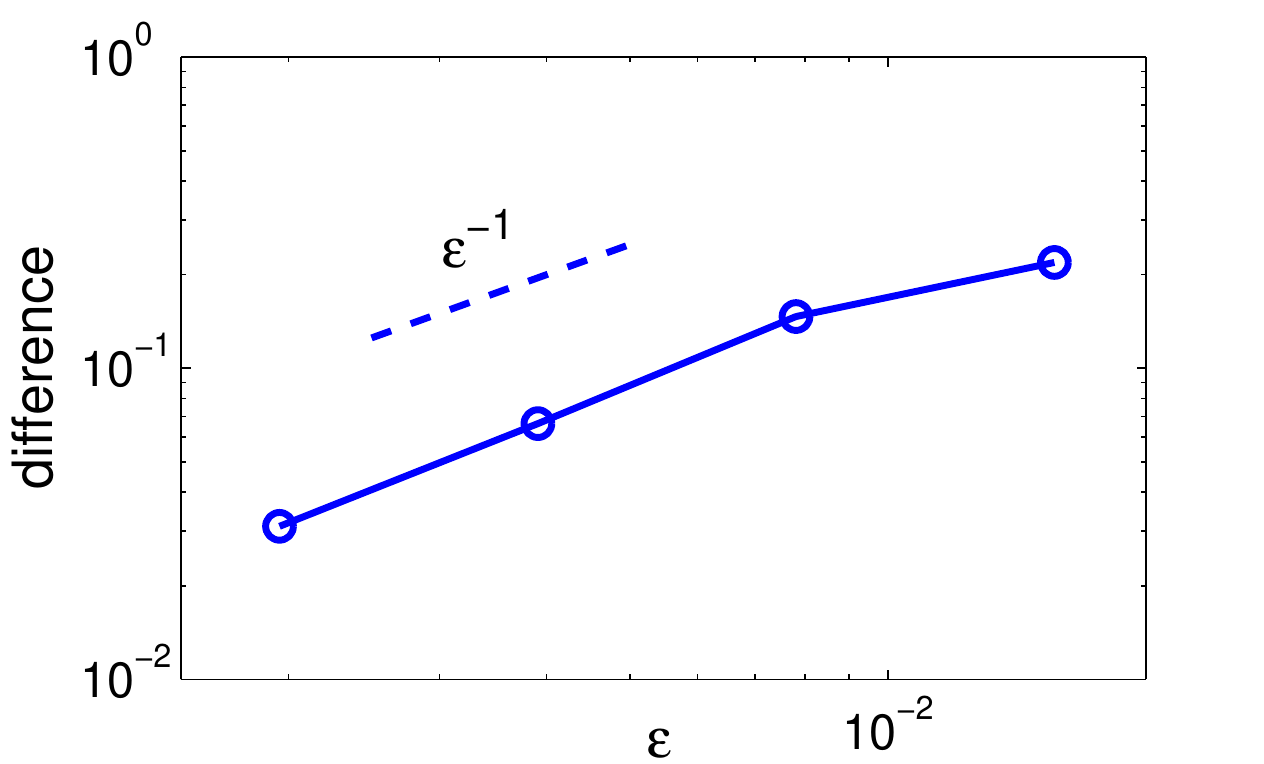,height=4cm,width=7cm}
\end{array}$$
\caption{Vlasov-Poisson case. Left: time history of the energy error with initial condition \ref{KH}. Right:
difference between (\ref{eq:1}) and the limit model (\ref{eq:lm}) (maximum error of $|\rho^\eps(t=\pi,\bx)-\rho(t=\pi,\bx)|/|\rho^\eps(t=\pi,\bx)|$).}\label{fig:energy3d}
\end{figure}



\appendix

\section*{Acknowledgements}
This work is supported by the French ANR project MOONRISE ANR-14-CE23-0007-01. This work has been carried out within the framework of the French Federation for Magnetic Fusion Studies (FR-FCM) and of the Eurofusion consortium, and has received funding from the Euratom research and training programme 2014-2018 and 2019-2020 under grant agreement No 633053. The views and opinions expressed herein do not necessarily reflect those of the European Commission.

\bibliographystyle{model1-num-names}

\end{document}